\documentclass[oneside,11pt]{amsart}

\usepackage{fullpage}

\makeatletter
\def\@tocline#1#2#3#4#5#6#7{\relax
	\ifnum #1>\c@tocdepth 
	\else
	\par \addpenalty\@secpenalty\addvspace{#2}%
	\begingroup \hyphenpenalty\@M
	\@ifempty{#4}{%
		\@tempdima\csname r@tocindent\number#1\endcsname\relax
	}{%
		\@tempdima#4\relax
	}%
	\parindent\z@ \leftskip#3\relax \advance\leftskip\@tempdima\relax
	\rightskip\@pnumwidth plus4em \parfillskip-\@pnumwidth
	#5\leavevmode\hskip-\@tempdima
	\ifcase #1
	\or\or \hskip 1em \or \hskip 2em \else \hskip 3em \fi%
	#6\nobreak\relax
	\hfill\hbox to\@pnumwidth{\@tocpagenum{#7}}\par
	\nobreak
	\endgroup
	\fi}
\makeatother

\usepackage{amsmath}
\usepackage{amssymb}
\usepackage{amsthm}
\usepackage[alphabetic]{amsrefs}
\usepackage{array}
\usepackage[all]{xy}
\usepackage{fancyhdr}
\usepackage{euscript}
\usepackage{graphics,graphicx}
\usepackage{cancel}
\usepackage{fancybox}
\usepackage{verbatim}
\usepackage[dvipsnames]{xcolor}
\usepackage{tikz-cd}
\usepackage{longtable}

\usepackage[tableposition=below]{caption}
\usepackage[colorlinks=false,hyperindex]{hyperref}

\hypersetup{
	colorlinks,
	citecolor=violet,
	linkcolor=violet,
	urlcolor=blue}

\newtheoremstyle%
{custom}%
{}
{}
{}
{}
{}
{.}
{ }
{\thmname{}
\thmnumber{}%
\thmnote{\bfseries #3}}%

\newtheoremstyle%
{Theorem}%
{}%
{}%
{\itshape}%
{}%
{}%
{.}%
{ }%
{\thmname{\bfseries #1}%
\thmnumber{\;\bfseries #2}%
\thmnote{\;(\bfseries #3)}}%

\theoremstyle{Theorem}
\newtheorem{theorem}{Theorem}[section]

\newtheorem{proposition}[theorem]{Proposition}
\theoremstyle{definition}
\newtheorem{definition}[theorem]{Definition}

\newtheorem{example}[theorem]{Example}
\newtheorem{remark}[theorem]{Remark}

\newtheorem{question}[theorem]{Question}

\newcommand{\ZZ}{\mathbb{Z}}
\newcommand{\QQ}{\mathbb{Q}}
\newcommand{\RR}{\mathbb{R}}
\newcommand{\CC}{\mathbb{C}}

\newcommand{\HH}{\mathbb{H}}

\newcommand{\NN}{\mathbb{N}}

\newcommand{\Ocal}{\mathcal{O}}
\newcommand{\SO}{\operatorname{SO}}

\newcommand{\Isom}{\operatorname{Isom}}
\newcommand{\Stab}{\operatorname{Stab}}

\newcommand{\Hcal}{\mathcal{H}}

\newcommand{\quat}[2]{\left( \frac{#1}{#2}\right) }
\newcommand{\Clf}{\operatorname{Clf}}

\renewcommand{\Im}{\operatorname{Im}}

\renewcommand{\Vec}{\operatorname{Vec}}

\newcommand{\SL}{\operatorname{SL}}
\newcommand{\PSL}{\operatorname{PSL}}

\renewcommand{\O}{\operatorname{O}}

\newcommand{\orb}{\operatorname{orb}}

\newcommand{\mon}{\vartriangleright}
\newenvironment{psmallmatrix}
{\left(\begin{smallmatrix}}
	{\end{smallmatrix}\right)}

\newcommand{\rherm}{\operatorname{rherm}}

\begin{document}

\title{Ford Spheres in the Clifford-Bianchi Setting 
}
\author{Spencer Backman, Taylor Dupuy, Anton Hilado, Veronika Potter}

\maketitle

\begin{abstract}
We define Ford Spheres $\mathcal{P}$ in hyperbolic $n$-space associated to Clifford-Bianchi groups $\PSL_2(O)$ for $O$ orders in rational Clifford algebras associated to positive definite, integral, primitive quadratic forms. 
For $\Hcal^2$ and $\Hcal^3$ these spheres correspond to the classical Ford circles and Ford spheres introduced by Ford.

We prove the Ford spheres are an integral packing.
If we assume that $O$ is Clifford-Euclidean then $\mathcal{P}$ is also connected.
We also give connections to Dirichlet's Theorem and Farey fractions.

In a discussion section, we pose some questions related to existing packings in the literature.
\end{abstract}

\tableofcontents

\section{Introduction}
In this paper we construct Ford spheres $\mathcal{P}$ associated to orders in Clifford algebras in hyperbolic $n$-space $\Hcal^n$ for every $n$ and exhibit their basic packing properties (see \S\ref{S:ford-spheres}). 

\begin{figure}[htbp!]
	\begin{center}
		\includegraphics[scale=0.33]{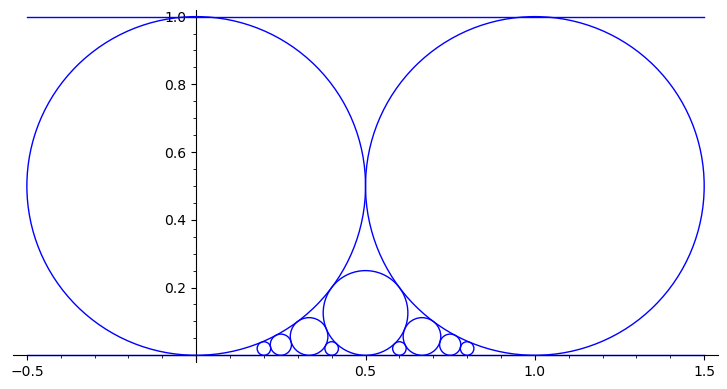}
	\end{center}
	\caption{A picture of the Ford circles in $\Hcal^2$.} \label{F:ford-circles}
\end{figure}

For us, hyperbolic $n$-space $\Hcal^n$ is the usual constant curvature Riemannian manifold and we will work exclusively in a half-space model (see \cite[\S4.6]{Ratcliffe2019}). 
By ``sphere'' we mean ``generalized sphere'' which are also sometimes called ``horospheres''. 
A \emph{horosphere} in $\Hcal^n$ is a Euclidean sphere tangent to the boundary of the half-space model or a codimension 1 Euclidean plane parallel to the boundary (again see \cite[\S4.6]{Ratcliffe2019}). See Figure~\ref{F:ford-circles} for a picture of the Ford circles in $\Hcal^2.$

Sphere packings in hyperbolic spaces are not new. 
The main contribution is showing that it is possible to directly generalize the techniques in dimensions 2 and 3 using fractional linear transformations to all dimensions.
Our hyperbolic spaces $\Hcal^n$ are built out of a ring of Clifford numbers (see \S\ref{S:clifford-algebras}) 
$$\CC_n = \RR[i_1,i_2,\ldots, i_{n-1}].$$ 
The elements $i_a$ satisfy $i_a^2=-1$ and $i_ai_b=-i_bi_a$ for $0<a<b<n$.
As a set $\Hcal^n$ a subset of a distinguished $\RR$-vector subspace 
$$V_n \cong \RR^n,$$ 
of $\CC_n$ called the \emph{Clifford vectors}.
These are elements of the form $x=x_0+x_1i_1 + \cdots + x_{n-1}i_{n-1}$ for $x_j \in \RR$.
The elements of $\Hcal^n$ are those Clifford vectors where $x_{n-1}>0$.
This uniformization of hyperbolic space affords us a representation of the group of orientation preserving isometries $\Isom(\Hcal^n)^{\circ}$ by fractional linear transfomations 
$$x\mapsto (ax+b)(cx+d)^{-1}.$$ Here $\begin{psmallmatrix} a& b\\c& d\end{psmallmatrix} \in \PSL_2(\CC_{n-1})$.
This fractional linear representation in turn allows us to consider the action of $\PSL_2(O)$ for $O \subset \CC_n$ an order in a rational Clifford algebra $K$.\footnote{The technical condition is that $K$ a  $\QQ$-Clifford algebra associated to a positive definite integral primitive quadratic form $q$. 
These rational Clifford algebras $K$ are generalizations of imaginary quadratic fields.}
The groups $\PSL_2(O)$ and $\SL_2(O)$ are called \emph{Clifford-Bianchi groups}.

In the case $n=2$, we have $\CC_1\cong\RR$ and the collection $\mathcal{P}$ are the famous Ford circles in $\Hcal^2$ from \cite{Ford1916, Ford1938} (Figure~\ref{F:ford-circles}). 
The circle tangent to the reduced fraction $p/q \in \QQ$ has radius $1/(2q^2)$.
The rational Clifford algebra $K$ will be $\QQ$, the order $O$ must be $\ZZ$, and the Clifford-Bianchi group must be $\PSL_2(\ZZ)$.

In the case $n=3$, we have $\CC_2\cong\CC$ and the collection $\mathcal{P}$ will agree with Ford Spheres in $\Hcal^3$ \cite{Ford1918,Ford1925} (Figure~\ref{F:ford-gaussian}).\footnote{In this uniformization $\Hcal^3 \subset \RR + \RR i + \RR j = V_3$ and $\CC_3=\HH=\RR[i,j]$ are Hamilton's quaternions. }
The rational Clifford algebra $K$ will be an imaginary quadratic field and the order $O$ will be an order in an imaginary quadratic field. 
The Clifford-Bianchi groups in this case coincide with usual Bianchi groups.

\begin{figure}[htbp!]
	\begin{center}
		\begin{tabular}{cc}
			\includegraphics[scale=0.5,trim = 4cm 3cm 0 0]{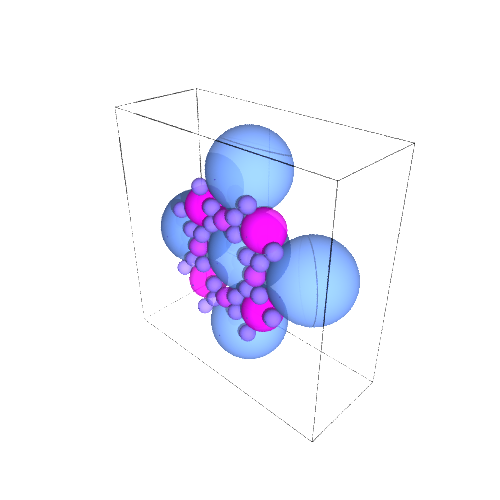} & \includegraphics[scale=0.5,trim = 0 0 0 0]{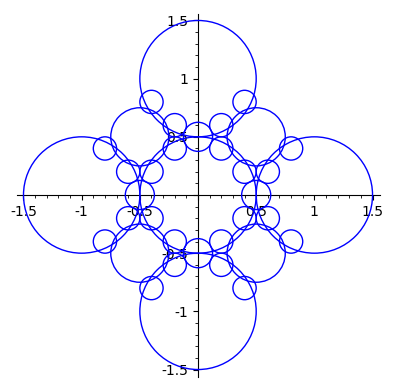} \\
		\end{tabular}
	\end{center}
	\caption{The Ford packing for $\PSL_2(\ZZ[i])$. Left: A picture of the Ford packing in $\Hcal^3$. Right: the projection of the equators of the Ford packing. \label{F:ford-gaussian}}
\end{figure}

\begin{figure}[htbp!]
	\begin{center}
		\begin{tabular}{cc}
			\includegraphics[scale=0.5,trim = 4cm 3cm 0 0]{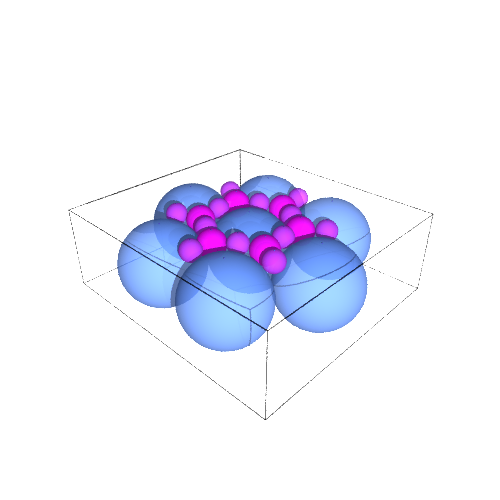} & \includegraphics[scale=0.5,trim = 0 0 0 0]{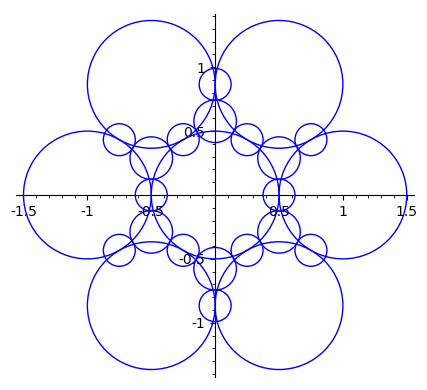}\\
		\end{tabular}
	\end{center}
	\caption{The Ford packing for $\PSL_2(\ZZ[\omega])$ where $\omega^2+\omega+1=0$. Left: A picture of the Ford packing in $\Hcal^3$. Right: the projection of the equators of the Ford packing. \label{F:ford-eisenstein}}
\end{figure}

In the case $n=4$, we have $\CC_3\cong \HH$, where $\HH=\RR[i,j]$ is Hamilton's quaternions.
The group $\PSL_2(\HH)$ acts on $\Hcal^4$.
If $\Ocal_3$ is the Hurwitz order $\ZZ[i,j,\frac{1+i+j+ij}{2}]$ then we expect $\mathcal{P}$ to coincide with a named special packing in $\Hcal^4$ in the literature although we were unable to find a reference for this packing.

In the cases $n>4$, the structure of $\CC_{n-1}$ is more complicated and governed by Bott periodicity for Clifford algebras (\cite[\S2.11]{Dupuy2024}).\footnote{We warn the reader that the Clifford numbers $\CC_4$ are \emph{not} the octonians and contain zero divisors.}

The orders in these Clifford algebras simultaneously generalize orders in imaginary quadratic fields and orders quaternion algebras $\quat{-a,-b}{\QQ}$ for $a,b$ positive coprime integers. 
These orders appear to be of fundamental interest and we know very little about the basic algebraic number theory governing their behavior.\footnote{
	See \cite{Sheydvasser2018,Sheydvasser2019} for a discussion in the quaternionic case.
}
For example we know in $\CC_8$ there is an order $O_{E_8}$ whose lattice of Clifford vectors is an $E_8$-lattice. 
Also, in $\CC_4$ there is an order $\Ocal_4$ whose lattice of Clifford vectors is a $D_4$-lattice and whose conjugation action witnesses the famous triality automorphism of its Dynkin diagram  (see \cite[\S13 -\S16]{Dupuy2024} and Example~\ref{E:orders} for more orders; in the quaternionic case see \cite{Sheydvasser2018}).
We do not know if there exists infinitely many Clifford-Euclidean such orders (see \S\ref{S:clifford-euclidean} for a definition).

For brevity we isolate commentary on history, context, and connections to other constructions in the literature to a discussion section (\S\ref{S:discussion}). 

We now state our results.
Following \cite[\S 3]{Maxwell1982}, by a \emph{ball packing}, we mean an infinite collection of balls in a metric space whose interiors are disjoint. 
By a \emph{sphere packing} we mean the spheres of a ball packing. 
A packing is called \emph{integral} if the reciprocal radii (curvatures) of all of the spheres are integers in the Clifford uniformization.\footnote{Some authors call the curvatures the \emph{bends} of the spheres.}
\begin{theorem}\label{T:main}
For every order $O \subset \CC_n$ associated to a positive definite integral primitive quadratic form $q$, the associated Ford spheres $\mathcal{P}$ are an integral packing in $\Hcal^{n+1}$.
\end{theorem}
Just as in the Bianchi setting, the general geometry of the Ford spheres seems to be controlled by the arithmetic of the Clifford-Bianchi group $\SL_2(O)$.

We say that a packing is \emph{connected} if its associated tangency graph is connected; in this graph the vertices are the spheres and there is an edge between two spheres if they are tangent. 
\begin{theorem}
If $O$ is Clifford-Euclidean then its associated collection of Ford spheres $\mathcal{P}$ is connected.
\end{theorem}
In the connected case we give a proof of an analog of Dirichlet's Theorem on approximation of irrationals by rationals.
We also provide some commentary of Farey fractions and their relationship to adjacent Ford spheres.

\subsection{Roadmap}

Internal disjointness is proved in \S\ref{S:internal-disjointness}, integrality is proved in \S\ref{S:integrality}, and connectedness under the Euclidean hypothesis is proved in \S\ref{S:connectivity}.
While abstract theories for sphere packings and symmetric spaces exist (see \cite[\S 5.5]{Dupuy2024}) we emphasize that our proofs employ M\"obius transformation.
The proofs depend on understanding the action $\Gamma=\SL_2(O)$ on the partial Satake compatification $\Hcal^{n+1,*}=\Hcal^{n+1} \cup (\Vec(K) \cup \lbrace \infty \rbrace)$, where $\Vec(K) = V_n \cap K$.

To prove internal disjointness we need to inscribe our spheres inside a union of fundamental domains under an orbit of a stabilizer at a point. 
This requires a detailed description of fundamental domains for these groups which we provide.

\subsection*{Acknowledgements} 
Veronika Potter was supported by the University of Vermont's College of Engineering and Mathematical Sciences Undergraduate Research Initiative.
Spencer Backman was supported in part by a Simons Collaboration Gift (\# 854037) and the National Science Foundation (DMS-2246967). Taylor Dupuy is supported in part by the National Science Foundation (DMS-2401570).  

The authors would like to thank Asher Auel, Eran Assaf, David Dummit, Arthur Baragar, Colin Ingalls, Alex Kontorovich, Adam Logan, Senia Sheydvasser, Daniel Martin, and John Voight for useful conversations.

\section{Background}

\subsection{Clifford Algebras}\label{S:clifford-algebras}
Let $(W,q)$ be a quadratic space over a commutative ring $R$. (This is just a projective $R$-module with a quadratic form on it. For this paper one can just think of $R^m$ for some $m \in \NN$ and a map $q: R^m \to R^m$ induced by a bilinear form $B_q$ via $B_q(x,x) = q(x)$.
From the quadratic form we recover the bilinear form via $B_q(x,y) =( q(x+y)-q(x)-q(y))/2$ when $1/2 \in R$.
)
\begin{definition}
The \emph{Clifford algebra} of $(W,q)$ is $\Clf(W,q) = T(W)/I_q$ where $T(W)$ is the tensor algebra and $I_q$ is generated by the relations $w\otimes w = -q(w)$ for all $w\in W$.
\end{definition}
We will sometimes denote this as $\Clf(q)$.
The \emph{Clifford vectors} are the $R$-submodule generated by $R$ and $W$ inside $\Clf(W,q)$. 
We denote them by $\Vec(q)$ or $\Vec(\Clf(W,q))$.  

Any Clifford algebra comes with three involutions: 
\begin{itemize}
\item the parity involution $x\mapsto x'$, 
\item the transpose or reversal involution $x\mapsto x*$ 
\item clifford conjugation $x\mapsto \overline{x} = (x')^* = (x^*)'$. 
\end{itemize}
The partity involution is the unique involution on $\Clf(W)$ induced by the map $w \mapsto -w$ on $W$. 
The reversal involution takes products of basis vectors and reverses them and then extends linearly. 
Note that for all elements $\alpha,\beta$ of the Clifford algebra we have $(\alpha \beta)'= \alpha'\beta'$ and $(\alpha\beta)^* = \beta^* \alpha^*$ and $\overline{\alpha\beta} = \overline{\beta} \overline{\alpha}$.

The analog of the real number, complex numbers, and Hamilton's quaternions are the Clifford numbers which we now define. All of our rational clifford algebras and orders can be thought of as embedded into this fixed ring.
\begin{definition}
Let $f_n = y_1^2 + \cdots + y_n^2$. 
The \emph{Clifford numbers} $\CC_n$ are the Clifford algebra 
$$\CC_n=\Clf(\RR^{n-1},f_{n-1}) = \RR[i_1,i_2,\ldots,i_{n-1}], \qquad i_a^2=-1, \quad 1\leq a <n,$$ $$i_ai_b = -i_bi_a, \quad 0<a<b<n.$$
\end{definition}
\noindent The Clifford vectors $V_n = \Vec(\CC_n)$ are a special $n$-dimensional $\RR$-vector subspace of $\CC_n$.
\begin{definition}
	The \emph{Clifford vectors} $V_n \subset \CC_n$ are the collection of $x=x_0+x_1i_1 + \cdots + x_{n-1}i_{n-1}$ where $x_j \in \RR$.
\end{definition}
Given an element $x \in V_n$ we let $x_j$ be the coefficient of the element $i_j$ in its expansion in terms of its standard basis understanding that $i_0=1$. 
For $x\in V_n$ we have $\vert x \vert^2 = x_0^2+x_1^2 + \cdots +x_{n-1}^2 = x \overline{x}$.

We let the units of $\CC_n$ be $U(\CC_n)$. 
The \emph{Clifford group} $\CC_n^{\times}$ is the subgroup of the group of units of $\CC_n$, generated by Clifford vectors. 
These play an important role in the geometry of spin groups.
For the present paper, the reader just needs to know that these are the ``good analog'' of unit groups.

We now come to our analogs of quadratic number fields. Let $q = d_1 y_1^2 + \cdots + d_{n-1}y_{n-1}^2$ be a $\ZZ$-form of $f_{n-1}$ with $d_j$ square-free positive integers for $1\leq j \leq n-1$.
We will suppose that our form is primitive, meaning that the ideal generated by $q(w)$ for $q \in \ZZ^{n-1}$ will be the unit ideal.
Let $K = \Clf(q)$. 
We view $K$ as a subalgebra of $\CC_n$ via $K = \QQ[\sqrt{d_1} i_1, \ldots, \sqrt{d_{n-1}} i_{n-1}]$. 
We will use the Hilbert symbol notation for these orders 
$$\quat{-d_1,-d_2,\ldots,-d_{n-1}}{\QQ} = \QQ[\sqrt{d_1} i_1, \ldots, \sqrt{d_{n-1}} i_{n-1}].$$
We give some basic examples.
\begin{example}[{\cite[Example 2.2.2]{Dupuy2024}}]
\begin{enumerate}
	\item When $q$ is $1$-ary, i.e. $q(y) = dy^2$, then $K\cong \QQ(\sqrt{-d})$.  
	\item When $q$ is $2$-ary we get the quaternion algebras $K=(-d_1,-d_2/\QQ)$.
\end{enumerate}
\end{example}

We will consider $O$ an order in $K$.
We recall that these are subrings which are $\ZZ$-lattices in $K$ viewed as a $\QQ$-vector space.
In this paper we will always assume our orders to be closed under the involution $*$. 
This simplifying assumption allows us to get rid of the words ``left'' and ``right'' a lot of the time.

The \emph{Clifford vectors} of an order are defined to be 
$$\Vec(O) = O \cap V_n.$$
These are important lattices in $\RR^n$ which control a lot of the arithmetic of $O$.

The \emph{Clifford monoid} is defined to be $O^{\mon} = O \cap (\CC_n^{\times} \cup \lbrace 0 \rbrace)$.
The \emph{Clifford group} of the order $O^{\times}$ are the elements are the nonzero elements of $\O^{\mon}$ whose inverse is also in the order.
In the case of positive definite $O$, this is a finite group \cite[Proposition 3.2.5]{Dupuy2024}.

\subsection{Clifford-Euclidean Domains}\label{S:clifford-euclidean}
The algebraic number theory of orders in Clifford algebras appears to be fairly open. 
One notion that has been isolated and useful is the analog of Euclidean domains.
\begin{definition}\label{def:clifford-euclidean}
	Let $O$ be an order	in a Clifford algebra.
	We say that $O$ is {\em right Clifford-Euclidean} if there is a norm function $N: O^\mon \to \NN$ such that
	\begin{enumerate}
		\item $N(x) = 0$ if and only if $x$ is a zero-divisor.
		\item \label{I:division-algorithm} For all $x,y \in O^\mon$ with $N(x) > 0$ and $xy^* \in \Vec(R)$, there exists some $q\in \Vec(R)$ and some $r\in O^\mon$ such that 
		\begin{equation}\label{eqn:left-division}
		y=xq+r
		\end{equation}
		where $N(r)<N(x)$.
	\end{enumerate}
\end{definition}

There is an obvious notion of left Clifford-Euclidean domain as well. 
We will work with orders which are \emph{compatible} meaning they are closed under the $*$-conjugation.
We use the following criterion for producing Clifford-Euclidean examples: if $\Vec(O)$ is a lattice in $V_n$ with covering radius less than $1$ then the order is Clifford-Euclidean \cite[\S3.4.1]{Dupuy2024}. 
Such examples are called \emph{norm Clifford-Euclidean} (since the Euclidean norm is the reduced norm).
In \cite{Dupuy2024} the authors performed a search in Magma for orders in low dimension which are contained in $(-d_1,\ldots,-d_{n-1}/\QQ)$ and contain $(-d_1,\ldots,-d_{n-1}/\QQ)$, and found a number of new classes of orders.

\begin{example}\label{E:orders}
In what follows multiple subscripts on an $i$ mean a product of the corresponding elements.
For example $i_{12} = i_1i_2$ and $i_{123} = i_1i_2i_3$.
	\begin{enumerate}	
		\item The Hurwitz order inside the rational quaternions 
		$$O_3 = \ZZ[i,j,\frac{1+i+j+ij}{2}].$$
		The group $\SL_2(O_3)$ acting discretely and with finite covolume on hyperbolic $4$-space.
		\item In one higher dimension there is a unique maximal order in $(-1,-1,-1/\QQ)$ which we denote by 
		$$O_4 = \ZZ[i_1,i_2,i_3, \frac{1+i_1+i_2+i_3}{2}].$$
		This is a new higher dimensional order which is an analog of the Gaussian integers which was found by experimentation \cite[\S11.1]{Dupuy2024}. The group $\SL_2(O_4)$ acts discretely and with finite covolume on hyperbolic $5$-space.
		This order has a remarkable connection to triality of the $D_4$-lattice.
		\item In $(-1,-1,-1,-1/\QQ)$ there are two classes of orders, one containing five orders and another containing a unique order. 
		The [5,1,4]-orders associated to [5,1,4] binary codes are Clifford-Euclidean (5 in its class) while the oddball order $\Ocal_{5,!}$ is not Clifford-Euclidean. 
		The [5,1,4]-orders are described in \cite[\S15.2]{Dupuy2024}. For example the order $\Ocal_{5,2}$ corresponding to $11011$ is takes the form
		 $$ \Ocal_{5,2} = \ZZ[i_1,i_2,i_3,i_4]\left [ \frac{i_0+i_1+i_3+i_4}{2}\right ].$$
		\item In $(-1,-3/\QQ)$ there are the orders $\Ocal(-1,-3)_i$ for $i=1,2$ which are described in \cite[\S12]{Dupuy2024}. These are Clifford-Euclidean. 
		\item In $(-1,-1,-3/\QQ)$ there are two classes of orders $A(-1,-1,-3)$ and $B(-1,-1,-3)_j$ for $j=1,2,3$. 
		The orders $B(-1,-1,-3)_j$ are Clifford-Euclidean while $A(-1,-1,-3)$ is not. 
		The geometry of the orders $B(-1,-1,-3)_j$ is described in \cite[\S14.1]{Dupuy2024}. They take the form
		 $$ B(-1,-1,-3)_j = \ZZ[i_1,i_2,\sqrt{3}i_3]\left [\frac{i_j+\sqrt{3} i_{j3}}{2} \right].$$
		 
	\end{enumerate}
\end{example}
The manuscript \cite{Dupuy2024} also reports on an order $\Ocal_{E_8}$ whose underlying lattice of Clifford vectors was an $E_8$ lattice which is very close to Clifford-Euclidean in that its covering radius is exactly $1$. 
The order from \cite{Dupuy2024}, with exception to the quaternion orders, are new.

\begin{remark}
Presently it is an open problem to determine if there are infinitely many Clifford-Euclidean orders (see \S16.1 of \cite{Dupuy2024}).
\end{remark}

\begin{example}
	Sheydvasser has investigated the higher dimensional class number one problem in the quaternionic case
	\cite{Sheydvasser2023,Sheydvasser2021,Sheydvasser2019,Sheydvasser2018}.
	Quaternion orders which are norm Clifford-Euclidean are investigated in \cite{Sheydvasser2019} under the name of ``covered by unit balls'' (see loc. cit. just below Theorem 11.1), and Clifford-Euclidean orders are called $\sigma$-Euclidean in \cite{Sheydvasser2023,Sheydvasser2021}.
	Sheydvasser. 
	In \cite[Theorem 13.2]{Sheydvasser2019}, he gives a finite list of rational quaternion algebras which contain Clifford-Euclidean orders and classifies these orders.
	Also \cite[Corollary 12.1]{Sheydvasser2019} gives many examples which are provably not Clifford-Euclidean (not just not norm Clifford-Euclidean). 
	
	We warn the reader that the involution $\sigma$ in Sheysvasser's examples is not always equal to $\sigma = *$, the Clifford transpose/reversal involution, which is always the case in our examples.
\end{example}

\subsection{Hyperbolic Geometry}
The Clifford Uniformization of hyperbolic space is Riemannian manifold $\Hcal^{n+1}$ defined by 
 $$ \Hcal^{n+1} = \lbrace x \in V_{n+1} \colon x_n > 0 \rbrace, \qquad ds^2 = \frac{dx_0^2 + dx_1^2 + \cdots + dx_n^2}{x_n^2}. $$
The orientation preserving isometries all take the form $x \mapsto (ax+b)(cx+d)^{-1}$ for $\begin{psmallmatrix} a & b \\ c & d \end{psmallmatrix} \in \SL_2(\CC_n).$
Here the group $\SL_2(\CC_n)$ is defined by 
 $$ \SL_2(\CC_n) = \lbrace \begin{psmallmatrix} a & b \\ c& d \end{psmallmatrix} \in M_2(\CC_n^{\times} \cup \lbrace 0 \rbrace) \colon ad^*-bc^* =1, ac^{-1}, bd^{-1} \in V_n \rbrace.$$
For $O$ an order in a positive definite Clifford algebra we define $\SL_2(O)$ to be the elements $g$ with of $g,g^{-1} \in M_2(O) \cap \SL_2(\CC_n)$.
These can be identified as arithmetic subgroups of $\SO_{1,n+1}(\RR)^{\circ} \cong \Isom(\Hcal^{n+1})^{\circ}$.
The connection with the isometry groups is given in \cite[\S5.4]{Dupuy2024} and arithmeticity is proved in \cite[Theorem 6.1.4]{Dupuy2024}.

There exists an extension of Bianchi-Humbert Theory which describes the fundamental domain $D$ of $\SL_2(O)$ acting on $\Hcal^{n+1}$. 
To do this we need to talk about the partial Satake compactification.
\begin{definition}
	The \emph{partial Satake compactification} of hyperbolic $(n+1)$-space with respect to a rational Clifford algebra $K$ associated to a positive definite primitive quadratic form is 
 $$ \Hcal^{n+1,*} = \Hcal^{n+1} \cup \Vec(K) \cup \lbrace \infty \rbrace.$$
\end{definition}
Note that $\SL_2(O)$ acts on $\Vec(K) \cup \lbrace \infty \rbrace$. 
The full compactification we take to be $\overline{\Hcal}^{n+1} = \Hcal^{n+1} \cup V_n \cup \lbrace \infty \rbrace$ where $V_n \cup \lbrace \infty \rbrace$ is the one-point-compactification which topologically just the $n$-sphere $S^n$.
A detailed treatment of this can be found in \cite[\S5.6]{Dupuy2024}.

In order to talk about ``rational points'' of the boundary we need to talk about what it means to be a fraction.
Suppose that $O$ is stable under $*$.
``Unimodularity'' is the Clifford analog of saying that $\mu^{-1}\lambda$ is a fraction in lowest terms.
\begin{definition}
We say that $(\mu,\nu) \in O^2$ is \emph{(right) unimodular} if and only if $\exists \begin{psmallmatrix} * & * \\ \mu & \nu \end{psmallmatrix}\in \SL_2(O)$. 
\end{definition}
This implies that $\mu,\nu \in O^{\mon}$ and that $\mu^{-1}\nu \in \Vec(K) \cup \lbrace \infty \rbrace$ \cite[\S 5]{Dupuy2024}.
We call a pair $(\mu,\nu) \in O$ \emph{left unimodular} provided there is some matrix $\begin{psmallmatrix} \mu & * \\ \nu & * \end{psmallmatrix}$.
Left unimodularity means that there exists some $g$ such that $g(\infty)=\mu\nu^{-1}$ makes sense.

\begin{remark}
	The above definition should be compared to \cite[Definition 13.1]{Sheydvasser2019} where a similar definition is used in the quaternionic case.
\end{remark}

We now describe the fundamental domain for $\PSL_2(O)$ acting on $\Hcal^{n+1}$.
 A discussion of the fundamental domain and algorithms is given in \cite[\S 7]{Dupuy2024}.
\begin{definition}\label{D:bubble}
	The \emph{bubble domain} is the set 
	$$B=\lbrace x \in \Hcal^{n+1} \colon \forall (\lambda,\mu),  \vert x -\mu^{-1}\lambda\vert \geq 1/\vert\mu\vert \rbrace$$
	where $(\lambda,\mu)$ run over unimodular pairs.  
\end{definition}

Fundamental domains for $\PSL_2(O)$ are then the collection of point in the bubble domain which lie above a fundamental domain for the stabilizer of $\infty$ at the boundary.
\begin{theorem}[{\cite[\S 7]{Dupuy2024}}]
	The fundamental domain for $\SL_2(O)$ is the set of $y+ix_n \in \Hcal^{n+1}$ where $y \in F \subset V_n$ a fundamental domain for $\SL_2(O)_{\infty} = \Stab_{\SL_2(O)}(\infty)$ the stablizer of $\infty$ and $y+ix_n\in B$. 
\end{theorem}

We will use the geometry of the fundamental domains, and in particular the bubble domain, to discuss our sphere packings. The basic idea is that $S_{\infty}$ is inscribed in $\overline{B}$ so that orbits of $\overline{B}$ then correspond to the packing $\mathcal{P} = \orb_{\PSL_2(O)}(S_{\infty})$.

\subsection{Examples in $\Hcal^3$}

\begin{figure}[htbp!]
	\begin{center}
		\begin{tabular}{cc}
			\includegraphics[scale=0.5,trim = 4cm 3cm 0 0]{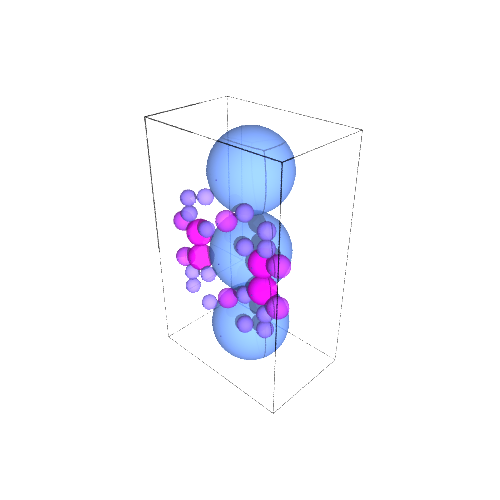} & \includegraphics[scale=0.5,trim = 0 0 0 0]{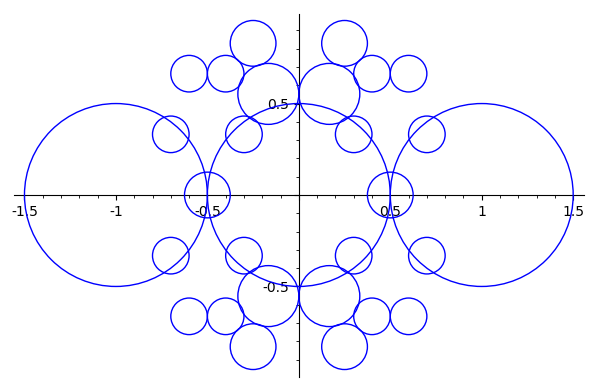} \\
		\end{tabular}
	\end{center}
	\caption{The Ford packing for $\PSL_2(\ZZ[i\sqrt{11}])$. Left: A picture of the Ford packing in $\Hcal^3$. Right: the projection of the equators of the Ford packing. \label{F:ford-sqrtm11}}
\end{figure} 

In terms of sphere packings, one can then define a sphere at infinity by the formula $S_{\infty} = V_2+i_3$ and then consider the orbit of this space under the Bianchi action. 
These packings are described, for example, in \cite{Northshield2015}.
We will call these Ford spheres in three dimensional space. 
The picture for $\SL_2(\ZZ[i])$ is shown in  \ref{F:ford-gaussian}. 

\begin{itemize}
	\item Figure~\ref{F:ford-gaussian} shows the packing for $O = \ZZ[i]$.
	\item Figure~\ref{F:ford-eisenstein} shows this packing for $O=\ZZ[\omega]$ where $\omega^2+\omega+1=0$ is a primitive cube root of unity.
	\item Figure~\ref{F:ford-sqrtm5} shows the packing for $\ZZ[\sqrt{-5}]$ which is an example where the domain is not a UFD. 
	\item Figure~\ref{F:ford-sqrtm11} shows a Ford packing for $\ZZ[\sqrt{-11]}$. 
\end{itemize}

 \subsection{Ford circles in $\Hcal^2$}
 What follows supplements a classical treatment of Ford circles in Conway's \emph{The Sensual Quadratic Form} \cite[pages 27--33]{Conway1997}.
 The treatment here is included so that the readers can compare the classical construction in $\Hcal^2$ for comparison with our higher dimensional Clifford-algebraic construction.
We make use of the Satake compactification 
 $$ \Hcal^{2,*}=\Hcal^2 \cup \QQ \cup \lbrace \infty \rbrace.$$ 
We will always introduce rational numbers $p/q$ in reduced form. 
We use unimodularity to replace ``reduced form'' in the generalization. 
The \emph{Ford circles} are the collection of circles 
\begin{equation}\label{E:ford-circles}
\mathcal{P} = \lbrace S_r \colon r \in \QQ \cup \lbrace \infty \rbrace \rbrace
\end{equation}
where if $r=p/q \in \QQ\setminus \lbrace 0 \rbrace$ then $S_{p/q}$ is the unique circle tangent to $p/q \in \QQ$ contained in $\Hcal^{2,*}$ of radius $1/2q^2$.
If $r=0$ then $S_0=S_{0/1}$ is the unique circle tangent to $0\in \QQ$ of radius $1/2$.
If $r=\infty$ we define 
$$S_{\infty} = \{z_0 + z_1 i \in \mathbb{C} \colon z_1 = 1\}.$$
We interpret $S_{\infty}$ as a circle at $\infty$ of infinite radius. 

The Ford circles are intimately tied to Farey fractions and can be viewed as the orbit of $\infty$ under $\SL_2(\ZZ)$ M\"{o}bius transformations, which we now review.

\begin{theorem}\label{tangent}
 	The Ford circles are internally disjoint.  Furthermore, let $p/q, r/s \in \mathbb{Q}$ be reduced, then $S_{p/q}$ and $S_{r/s}$ are tangent at some point if and only if $|pr-qs| = 1$.
\end{theorem}
 
 \begin{proof}
 	Let $p/q, r/s \in \mathbb{Q}$ be reduced, then the centers of the Ford circles $S_{p/q}$ and $S_{r/s}$ are $p/q+i/2q^2$ and $r/s+i/2s^2$, respectively.  The circles $S_{p/q}$ and $S_{r/s}$ are internally disjoint if and only if their centers are at least the sum of their radii away from each other.  We calculate that the distance 
 	
 	$$\left({ps-qr \over sq}\right)^2+\left({1/2q^2-1/2s^2} \right)^2  \geq \left( 1/2q^2+1/2s^2 \right)^2 \iff  {(ps-qr)^2 \over s^2q^2 } \geq {1\over s^2q^2} \iff  (ps-qr)^2 \geq 1 .$$
 	
 	This is true as $ps-qr$ is an integer.  Equality holds above precisely when the circles are tangent at a point and this occurs if and only if  $|ps-qr|=1$.  
 	
 \end{proof}
 
 We identify $\PSL_2(\mathbb{Z})$ with the set of integer M\"obius transformations $z \mapsto (az+b)(cz+d)^{-1}$ where $\begin{psmallmatrix} a & b \\ c & d \end{psmallmatrix} \in \SL_2(\ZZ)/\lbrace \pm I \rbrace$.  
 Given $A \in \PSL_2(\mathbb{Z})$ and a Ford circle $S_{p/q}$ for some $p/q \in \mathbb{Q}$, we take $A(S_{p/q})$ to be the image of $S_{p/q}$ under the M\"obius transformation associated to $A$.
 
 Before proceeding we recall that M\"obius transformations map the upper half plane to itself bijectively, and they map circles to circles (more generally this works in all dimensions \cite[\S 4]{Ratcliffe2019}, this is part of a general topic called inversive geometry which goes back to the 19th century -- see also loc. cit.). 
 
 We also need a general fact about geometry. 
 
 \begin{proposition}\label{L:line-proposition}
 	Let $u,v,w \in \mathbb{R}^n$.
 	Let $[u,v]$ be the directed line segment connecting $u$ and $v$.
 	Suppose that $w \in [u,v]$. 
 	Let $\vert u-w \vert = \alpha$ and $\vert v-w \vert = \beta$ so that $\vert u-v\vert = \alpha+\beta$.  
 	Then 
 	\begin{equation}\label{E:interpolation}
 	w={\beta \over \alpha + \beta}u + {\alpha \over \alpha + \beta} v.
 	\end{equation}
 \end{proposition}

 We now give an elementary proof that the Ford Circles as defined in \eqref{E:ford-circles} and the orbit of the sphere $S_{\infty}$ under $\PSL_2(\ZZ)$ coincide.
 \begin{theorem}\label{Mobdef}
 	The Ford circles $\mathcal{P}$ are the orbit of $S_{\infty}$ under $\SL_2(\ZZ)$.
 	As a consequence $\SL_2(\mathbb{Z})$ acts on the Ford circles by M\"obius transformations.
 	The action is given by $A \in \SL_2(\mathbb{Z})$ acts via 
 	 $$A(S_{p/q}) = S_{A(p/q)}.$$
 \end{theorem}
  \begin{proof}
  	Since $S_0$ and $S_{\infty}$ are related by $z\mapsto -z^{-1}$ we can (and will) show that every $S\in \mathcal{P}$ is a transformation of $S_0$.
 	
 	Let $A =
 	\begin{psmallmatrix}
 	r & p\\
 	s & q
 	\end{psmallmatrix} $
 	be an element of $\PSL_2(\mathbb{Z})$.  We claim that $A(S_0) = S_{p/q}$.  We first observe that $A(0) = p/q$, hence the image of $S_0$ under $A$ is tangent to the real line.  
 	Any circle in the upper half plane which is tangent to the real line is determined by its real point and one other point.
 	(This is because as you expand the radius for the circle tangent to the point on the real line, as it soon as it intersects another point it is determined --- the family is a one parameter family with that parameter being the radius.
 	A single condition fixed that parameter.)
 	
 	Therefore it suffices to check that the point $i \in S_0$ is mapped to $A(i) \in S_{p/q}$.  Notice that the circle at infinity $S_{\infty}$, which is the horizontal line $\Im(z) = 1$, is tangent to $S_0$ at the point $i$.  Moreover $A(\infty) = r/s$.  Therefore we know that  $A(S_0)$ and $A(S_{\infty})$ are both circles which are tangent to the real line and contain the point $A(i)$.  For proving that $A(i) \in S_{p/q}$, we prove the stronger statement that $A(i) \in S_{p/q} \cap S_{r/s}$.  
 	
  We apply equation \eqref{E:interpolation} with $u = p/q + i/2q^2$ and $v = r/s +i/2s^2$, which are the centers of the $S_{p/q}$ and $S_{r/s}$, respectively.  By Theorem \ref{tangent} we know that these circles are tangent, thus the distance between $u$ and $v$ is $1/2q^2+1/2s^2$.    Let $\alpha = 1/2q^2$, $\beta = 1/2s^2$, so that in \eqref{E:interpolation} $w$ is the their point of tangency.
  We get  $$w = {1/2q^2 \over 1/2s^2+1/2q^2}(r/s +i/2s^2)+ {1/2s^2 \over 1/2s^2+1/2q^2}(p/q + i/2q^2).$$  Next, we wish to show that $A(i) = w$ and calculate 
 	\begin{align*}
 	 	w &= {1/2q^2 \over (2q^2+2s^2) /4q^2s^2}(r/s +i/2s^2)+ {1/2s^2 \over (2q^2+2s^2) /4q^2s^2}(p/q + i/2q^2)\\
 	 	&= {s^2 \over q^2+s^2}(r/s +i/2s^2)+ {q^2 \over q^2+s^2}(p/q + i/2q^2)\\
 	 	&= {rs +i/2 \over q^2+s^2}+ {pq+i/2 \over q^2+s^2}= {rs + pq+i \over q^2+s^2}= {s+qi \over s+qi} \cdot {r-pi \over s-qi} = \frac{ri+p}{si+q} = A(i).
 	\end{align*}

 	Conversely, let $p/q \in \mathbb{Q}$ be in reduced form, so that $\gcd(p,q)=1$.
 	We may use the Euclidean algorithm to find $r/s \in \mathbb{Q}$ in reduced form such that $(r)q+(-s)p= 1$, and the M\"obius transformation associated to 
 	$\begin{psmallmatrix}
 	r & p\\
 	s & q
 	\end{psmallmatrix} $ maps $S_0$ to $S_{p/q}$. We conclude that $A(S_0)$ as $A$ ranges over all elements of $\SL_2(\mathbb{Z})$ is the set of all Ford circles.
 	
 	This establishes that all $\mathcal{P}= \orb_{\SL_2(\ZZ)}(S_{\infty})$.
 \end{proof}


For connectedness we will use the division algorithm.
We first observe that the set of Ford circles are topologically connected as a subset of the upper half plane if and only if the tangency graph of the Ford circles is graph connected. 
Here the tangency graph is defined with the vertices being the circles and that two circles are adjacent if and only if they are tangent. 
 \begin{theorem}\label{T:connected-sl2}
  	The set of all Ford circles is connected. 
 \end{theorem}
 \begin{proof}
 	For establishing connectedness of the dual graph of the Ford circles, if suffices to show that for any $p/q \notin \{0,\infty\}$, we have that $S_{p/q}$ is tangent to another Ford circle of larger radius.  By induction on the inverse of the radii of the spheres, which are always integers, it will follows that we may walk from $S_{p/q}$ to the sphere $S_{0/1}$ through a sequence of tangent Ford circles, and then by transitivity the dual graph of the collection of Ford circles is connected. 
 	
 	By Theorem \ref{tangent}, we must prove the existence of a rational $r/s$ such that $|ps-qr|=1$ and $s<q$.  By the Euclidean algorithm, we may find some $t/u$ such that $|pt-qu|= 1$.  Suppose without loss of generality that $q<t$ and let $aq +s= t$, with $a,s \in \mathbb{N}$ and $s < q$, then take $r = (u-ap)$.  We observe that $|ps -qr| = |p(t-aq)-q(u - ap)| = |pt -qu -apq + apq| = 1.$
 \end{proof}

\begin{remark}
It is important to note that in the proof above we are using a Euclidean algorithm to prove connectedness.
\end{remark}

Fix $d$ a denominator. 
The ordered set of 
$$\mathcal{F}_d = \lbrace a/b \in \QQ \cap [0,1] \colon a,b\in \NN, b<d \rbrace$$ are called the \emph{Farey Fractions} of level $d$. 
We will always want to consider fractions in reduced form.
For example when $d=5$ we have
$$ 0, \quad 1/5,\quad 1/4,\quad 1/3, \quad  2/5, \quad 1/2, \quad 3/5, \quad 2/3, \quad 3/4,\quad 4/5, \quad 1.$$
A classical fact is that $S_{p/q}$ and $S_{r/s}$ are adjacent if and only if $p/q$ and $r/s$ are adjacent in $\mathcal{F}_d$ for some $d$. 
Moreover if they are adjacent, the sphere $S_{(p+r)/(q+s)}$ where $(p+s)/(q+s)$ is their \emph{mediant} is mutually tangent to both.
It is also a fact that if $p/q<r/s$ are adjacent in some level then $\begin{psmallmatrix} r& p \\ s & q \end{psmallmatrix} \in \SL_2(\ZZ)$.
For example $\begin{psmallmatrix} 1 & 2 \\ 2 & 5 \end{psmallmatrix}$, $\begin{psmallmatrix} 2 & 1 \\ 5 & 3 \end{psmallmatrix} \in \SL_2(\ZZ)$.
The thing to observe is that if we let $g=\begin{psmallmatrix} r& p \\ s & q \end{psmallmatrix}$ then $g(0) = p/q$, $g(\infty) = r/s$, $g(1) = (r+p)/(s+q)$ and $g(-1) = (-r+p)/(-s+q)$ and that $S_0$ and $S_{\infty}$ are mutually adjacents that $S_{1}$ and $S_{-1}$ are tangent to them both. 
This means taking the image of these under $g$ preserves these adjacencies.
 
 \begin{theorem}
 	Let $p/q$, $r/s\in \QQ$ be reduced and suppose that $|ps-qr|=1$ so that $S_{p/q}$ and $S_{r/s}$ share a point of tangency.  Then $S_{p/q}$ and $S_{r/s}$ are both tangent to a Ford sphere $S_{\alpha}$ if and only if $\alpha$ is a mediant of $p/q$ and $r/s$.
 \end{theorem}
 
 \begin{proof}
 	Let $A = \begin{psmallmatrix}
 	r & p\\
 	s & q
 	\end{psmallmatrix}$ be an element of $\SL_2(\mathbb{Z})$.  
 	Tangency of Ford circles is preserved under action of $\SL_2(\mathbb{Z})$, and $A^{-1}(S_{p/q}) = S_0$ and $A^{-1}(S_{r/s}) = S_{\infty}$.  The Ford circles which are tangent to both $S_0$ and $S_{\infty}$ are $S_1$ and $S_{-1}$.  Therefore the circles tangent to  $S_{p/q}$ and $S_{r/s}$ are the circles $S_{A(1)} = S_{p + r /q + s}$ and $S_{A(-1)} = S_{p - r /q - s}$.
 \end{proof}

\section{Ford Spheres for $\PSL_2(O)$}\label{S:ford-spheres}
As is usual by ``spheres'' we mean horocycles --- this means either genuine spheres or codimension one hyperplanes. 
In particular we define the \emph{sphere at infinity} in $\Hcal^{n+1}$ we define to be 
$$S_{\infty}=V_n + i_n$$
Let $O$ be an order in $\quat{-d_1,\ldots,-d_{n-1}}{\QQ}$. 
We will let $\Gamma = \PSL_2(O)$.
\begin{definition}
	An element $x \in V_n = \partial\Hcal^{n+1}$ is call \emph{principal} if and only if $x \in \orb_{\Gamma}(\infty)$. 
	If $\gamma(\infty)=x$ the sphere $\gamma(S_{\infty})=S_x$ will be a called a \emph{principal sphere} or a \emph{Ford sphere}.
\end{definition}
It will be important to note that $S_{\infty}$ is inscribed in the bubble domain $B$ which is an orbit of the fundamental domain under the stabilizer of $\Gamma$ (Definition~\ref{D:bubble}).

\begin{remark}
	Note that this proves the statement about tiling in Theorem~\ref{T:main} of the introduction.
	The difficult part of this tiling claim is the explicit description of the fundamental domain which appears in \cite[\S7]{Dupuy2024} and which we are assuming.
\end{remark}

An image of $S_{\infty}$ inscribed inside the region $B$ is pictured in Figure~\ref{F:h2-spheres}.
\begin{figure}[h]
	\begin{center}
		\includegraphics[scale=0.5]{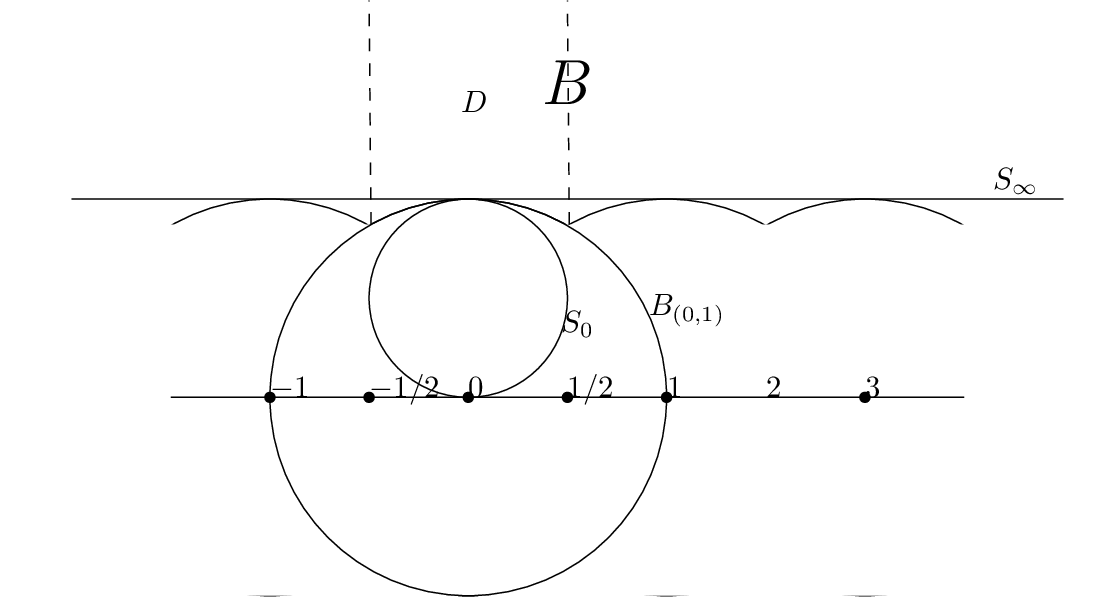}
	\end{center}
	\caption{A picture of $B$ with its inscribed sphere $S_{\infty}$ in $\Hcal^2$. }\label{F:h2-spheres}
\end{figure}

A sphere packing is usually defined to be a dense collection of spheres $\mathcal{P}$ where for every $S$ and $T$ in $\mathcal{P}$ the interior of the spheres don't intersect.
In contrast, our collection of spheres are internally disjoint, but are only dense on the boundary.
A collection of spheres is \emph{integral} if for any sphere $S$ in the collection the curvature of $S$ is an integer. 

Integrality is proved in \S\ref{S:integrality} and internal disjointness is proved in \S\ref{S:internal-disjointness}.
Internal disjointness relies on the construction of the fundamental domain (\cite[\S 7 ]{Dupuy2024}).
Figure~\ref{F:h2-spheres} shows the picture of $\Hcal^2$ where $S_{\infty}$ is inscribed in $B$.

\subsection{Reverse Clifford Hermitian Matrices}
This section should be compared to \cite[\S7]{Sheydvasser2019} where some similar computations are performed in the quaternionic case.

\begin{definition}
For a matrix $A=(a_{ij})$ in $M_{m,n}(\CC_{n})$ we define the \emph{reverse Clifford adjoint} (or simply \emph{reverse adjoint}) by 
		$$A=\begin{pmatrix}
			a_{11} & a_{12} & \cdots & a_{1n} \\
			a_{21} & a_{22} & \cdots & a_{2n} \\
			\vdots & \vdots & \ddots & \vdots \\
			a_{m1} & a_{m2} & \cdots & a_{mn}
		\end{pmatrix} \mapsto A^{\dag}=\begin{pmatrix}
		\bar{a}_{m n} & \bar{a}_{m-1n} & \cdots & \bar{a}_{1n} \\
		\bar{a}_{m n-1} & \bar{a}_{m-1 n-1} & \cdots  & a_{1 n-1}\\
		\vdots & \vdots & \ddots & \vdots \\
		\bar{a}_{m 1} & \bar{a}_{m-1 2} & \cdots & \bar{a}_{11}
		\end{pmatrix}$$
\end{definition}
This operation is an involution.
Also, observe that because $\Delta(A^{\dag}) = \overline{\Delta(A)}$ the group $\SL_2(\CC_n)$ is preserved under the Clifford adjoint.
We say $A \in M_2(\CC_n)$ is \emph{Clifford reverse Hermitian} (or simply \emph{reverse Hermitian} for convenience) if and only if it has Clifford vector entries and $A = A^{\dag}$. 

\begin{remark}
	This is in contrast to \cite[\S 7.2]{Dupuy2024} where Clifford-Hermitian matrices were needed. 
	This is interesting as there is some sort of duality between the Hermitian matrices and the Clifford-Hermitian matrices corresponding to the duality between Ford Spheres and Bubbles. 
	In terms of hyperbolic geometry and the Levi-decomposition, the Bubbles in the boundary of the fundamental domains for $\SL_2(O)$ described in \cite[\S7.7]{Dupuy2024} are geodesic and the Ford spheres are horoballs.
	See \cite[\S 5.7]{Dupuy2024} for more on this discussion.
\end{remark}

Note that if $A$ is reverse Hermitian then it has the form $A = \begin{psmallmatrix}
\beta & \alpha \\
\gamma & \overline{\beta}
\end{psmallmatrix},$ for $\alpha,\gamma \in \RR,$ and $\beta \in V_{n}.$
We will let the set of reverse Hermitian matrices be denoted by $M_2(\CC_n)^{\rherm}$.
Let $A\in R$. 
\begin{definition}
Its \emph{discriminant} or \emph{naive determinant} we denote by 
\begin{equation}\label{eqn:disc}
 \delta(A) = \vert \beta \vert^2 - \alpha \gamma.
\end{equation}
\end{definition}

\begin{definition}
The \emph{reverse Hermitian form} associated to $A$ is  
\begin{equation}\label{E:rherm-form}
q^A(z) = \alpha \overline{y}y + \overline{y}\beta x+\overline{x}\overline{\beta}y+\gamma \overline{x}x, \qquad z = \begin{pmatrix} x \\ y \end{pmatrix} \in \CC_n^2
\end{equation}
\end{definition}

If we write $z = \begin{pmatrix} x \\ y \end{pmatrix}$ we can also write this as $q^A(x,y) = z^{\dag}Az.$ 
Let $A\in M_2(\CC)^{\rherm}$ be a reverse Hermitian matrix. 
\begin{definition}
	The \emph{reverse Hermitian quadric} $S_A$ associated to $A$ is defined by the equation 
	\begin{equation}\label{E:quadric}
	S_A \colon \quad \gamma \vert x \vert^2 + \beta x + \overline{x} \overline{\beta} + \alpha =0, 
	\end{equation}
\end{definition}
We now show that all of these quadrics are spheres and compute their centers and radii.
\begin{theorem}\label{L:r-and-p}
	For every $A=\begin{psmallmatrix}
	\beta & \alpha \\
	\gamma & \overline{\beta}
	\end{psmallmatrix}\in M_2(\CC)^{\rherm}$ the reverse Hermitian quadric $S_A$ is a sphere of radius $r= \frac{\sqrt{\vert \beta\vert^2 - \alpha\gamma}}{\gamma}$ and center $p=-\overline{\beta}/\gamma.$
\end{theorem}

\begin{proof}
	By definition $S_A$ is given by the equation $\gamma \vert x \vert^2 + \beta x + \overline{x} \overline{\beta} + \alpha = 0$. This implies $\overline{x} x + \frac{\beta x}{\gamma} + \frac{\overline{x} \overline{\beta}}{\gamma}  = -\frac{\alpha}{\gamma}$.
	As $\frac{\beta}{\gamma} + \overline{x} = \overline{x + \frac{\overline{\beta}}{\gamma}}$, we can simplify to obtain $(\frac{\beta}{\gamma} + \overline{x})(x + \frac{\overline{\beta}}{\gamma}) - \frac{\vert \beta \vert^2}{\gamma^2} =  -\frac{\alpha}{\gamma}$,
	which implies 
	$$\vert x + \frac{\overline{\beta}}{\gamma} \vert^2 =  \frac{\vert \beta \vert^2 - \alpha \gamma}{\gamma^2},$$
	which gives our result.
\end{proof}

In \cite[\S 7.3]{Dupuy2024} it is shown that $\Hcal^{n+1}$ is isomorphic to classes of positive definite Clifford Hermitian matrices as $\SL_2(\CC_n)$-spaces.
Presently, we are going to relate actions on Ford spheres $\mathcal{P}$ to the action on \emph{reverse} Hermitian matrices $M_2(\CC_n)^{\rherm}$. 
The duality here is striking.

\begin{proposition}
	The Ford sphere at infinity $S_{\infty}$ is a reverse Hermitian quadric $S_{A_{\infty}}$ (see \eqref{E:quadric}). 
	That is there exists some $A_{\infty} \in M_2(\CC_n)^{\rherm}$ such that $S_{\infty}=S_{A_{\infty}}$.
\end{proposition}
\begin{proof}
	The sphere at infinity $S_{\infty}$ is defined by the equation $x_n-1=0$.
	From equation \eqref{E:quadric} we see that $\alpha=-1$ and using that $\beta x + \overline{x} \overline{\beta} = 2(\beta_0 x_0 - \vec{\beta}\cdot \vec{x})=x_n$ we see that $\beta=-i_n/2$ suffices and hence if  we define
	\begin{equation}\label{E:ainf}
	A_{\infty} := 
	\begin{pmatrix}
	i_n/2 & 1 \\
	0 & -i_n/2
	\end{pmatrix}
	\end{equation}
	we have $S_{A_{\infty}} = S_{\infty}$.
	The same equation could be done with $\beta = -i_n/2$ and $\alpha=-1$ and $\gamma=0$.
\end{proof}

\begin{proposition}\label{lem:invariance}
	\begin{enumerate}
		\item There is a well-defined action of $\PSL_2(\CC_n)$ on the set of reverse Hermitian matrices $M_2(\CC_n)^{\rherm}$ given by $A^g = g^{\dag} A g$.
		\item \label{I:conjugation-trick} $q^A(g^{\dag}z) = q^{A^g }(z)$.
	\end{enumerate}
\end{proposition}
\begin{proof}
It is enough to show that for $A \in M_2(\CC_n)^{\rherm}$ and $g \in \SL_2(\CC_n)$ that $(A^g)^{\dag} = A^g$. This follows directly from $(CB)^{\dag} = B^{\dag} C^{\dag}$ for all $C,B\in M_2(\CC_n)$.

For the second part we use that $q^A(z) =z^{\dag} A z $. 
This means that $q^A(gz) = (gz)^{\dag} A (gz) = z^{\dag} g^{\dag} A g z = q^{A^g}(z)$.
\end{proof}

We can now compute the equation of any sphere $S_A$ under M\"{o}bius transformations; before proceeding we note that the action of $g \in \SL_2(\CC_n)$ on $M_2(\CC_n)$ is given by $A \mapsto g^{\dag}Ag$ which is a right action, and the action on $V_n$ is given by M\"{o}bius transformations.
We are now ready to compute the image $g^{-1}(S_{\infty})$ for $g \in \PSL_2(\CC_n)$.
\begin{proposition}\label{L:image-of-sphere}
	If $g \in \SL_2(\CC_{n+1})$ and $A\in M_2(\CC_n)^{\rherm}$ then $g^{-1}(S_{A}) = S_{g^{\dag}Ag}$.
\end{proposition}
\begin{proof}
	The proof is just the elementary ``conjugation trick'' given in the proof of Proposition~\eqref{lem:invariance}, item \ref{I:conjugation-trick},
	\begin{align*}
	g^{-1}(S_A) &= \lbrace  g^{-1}x \colon q^A(x,1) =0 \rbrace=\lbrace y \colon q^A(gz)=0, z = (y,1) \rbrace=\lbrace y \colon  q^{g^{\dag}Ag}(z)=0, \ z=(y,1) \rbrace =S_{g^{\dag}Ag}.
	\end{align*}
\end{proof}	

\subsection{Integrality}\label{S:integrality}
In this section we prove integrality of $\mathcal{P}$, meaning we show that the curvatures (reciprocals of the radii) are integers.
\begin{theorem}\label{T:associated-A}
	Let $g = \begin{psmallmatrix} a & b \\ c & d \end{psmallmatrix} \in \SL_2(\CC_n)$.
	We have 
	$$g^{-1}(S_{\infty}) = S_A, \quad A = \begin{pmatrix}\beta & \vert d \vert^2 \\  \vert c \vert^2 &  \overline{\beta} \end{pmatrix}, \quad \beta = \overline{d}c + \frac{i_n}{2}(d^*a-b^*c).$$
\end{theorem}
\begin{proof}
	By Proposition~\ref{L:image-of-sphere} we have $g^{-1}(S_{\infty}) = S_{g^{\dag}A_{\infty}g}$ where $A_{\infty}$ is as in equation \eqref{E:ainf}. 
	Here we compute
	$$A_1:=\begin{pmatrix}
	\beta_1 & \alpha_1 \\
	\gamma_1 & \overline{\beta}_1 
	\end{pmatrix} := \begin{pmatrix} \overline{d} & \overline{b} \\
	\overline{c} &\overline{a} \end{pmatrix} 
	\begin{pmatrix}
	-i_n/2 & -1 \\
	0 & i_n/2
	\end{pmatrix}
	\begin{pmatrix}
	a & b \\
	c & d 
	\end{pmatrix} = g^{\dag}A_{\infty}g$$
	and find that 
	$$\alpha_1 =\vert d \vert^2 + \frac{i_n}{2}(d^*b-b^*d), \quad \beta_1 = \overline{d}c + \frac{i_n}{2}(d^*a-b^*c), \quad \gamma_1=\vert c \vert^2 + \frac{i_n}{2}(c^*a-a^*c).$$
	By definition of $\PSL(\CC_n)$, we know that $ab^*, cd^* \in V^{n+1}$. From \cite[Lemma 1]{Maclachlan1989} this means that $c^*a$ and $d^*b$ are also Clifford vectors. 
	We now simplify $\gamma_1$. For $z \in V^{n+1}$, we note that $z^* = z$. Further, we have that $(a^*c)^* = c^*a$. It follows that $$\gamma_1 = \vert c \vert^2 + \frac{i_n}{2}(c^*a-(c^*a)^*),$$
	As  $c^*a \in V^{n+1}$, $(c^*a)^* = c^*a$, this yields $\gamma_1 =  \vert c \vert^2.$
	We use this same technique to simplify $\alpha_1$ and obtain that
	$$\alpha_1 =\vert d \vert^2, \quad \beta_1 = \overline{d}c + \frac{i_n}{2}(d^*a-b^*c), \quad
	\gamma_1=\vert c \vert^2.$$
\end{proof}

In order to prove integrality we need a proposition on the discriminant.
\begin{proposition}\label{lem:conjugation-preserves}
	Let $g \in \PSL_2(\CC_{n})$ and let $A = \begin{psmallmatrix} \beta & \alpha \\ \gamma & \overline{\beta} \end{psmallmatrix}$ be reverse Hermitian.
	\begin{enumerate}
		\item  We have $g^{\dag}Ag$ is also reverse Hermitian.
		\item $\delta(g^{\dag}Ag) = \delta(A)$. 
	\end{enumerate}
\end{proposition}
\begin{proof}
	From Ahlfors, we have that $\PSL_2(\CC_{n})$ is generated by 
	$$X = \begin{pmatrix} a & 0 \\ 0 & a^{*-1} \end{pmatrix} , \quad Y=  \begin{pmatrix} 1 & b \\ 0 & 1 \end{pmatrix}, \quad Z =  \begin{pmatrix} 0 & 1 \\ -1 & 0 \end{pmatrix}$$
	where $a \in \Gamma_{n}$ and $b \in V^{n}$. 
	Hence every $g$ can be written as $g_1g_2\ldots g_m$ for some $m$ and $g_i$ each of the form above. 
	We will prove that statement by induction on $m \in \ZZ_{\geq 0}$. 
	The base case will be proved by checking that for each generator $g_0$ and each Clifford self-adjoint $A$ we have that $g_0^\dag A g_0$ is self-adjoint and that $\delta(g_0^{\dag}Ag_0)=\delta(A)$
	We will prove this at the end. 
	The inductive step follows since if $g=g_1g_2\ldots g_n=g_1 \widetilde{g}_1$ we have 
	$$ \delta( g^{\dag} A g ) = \delta(g_1^{\dag} (\widetilde{g}_1^{\dag} A \widetilde{g}_1) g_1) = \delta(\widetilde{g}_1^{\dag} A \widetilde{g}_1)= \delta(A).$$
	The first equality just factors the conjugation; the second equality uses that for $A$ Clifford self-adjoint and $\widetilde{g}_1 \in \SL_2(\CC_{n})$ the matrix $\widetilde{g}_1^{\dag} A \widetilde{g}_1$ remains self-adjoint together with the inductive hypothesis; the last equality follows from inductive hypothesis. 
	
	We now invariance of the discriminant under conjugation for each generator. 
	Using the formulas for matrix multiplication we find 
	$$X^{\dag} A X = \begin{pmatrix} \overline{a^{*-1}} \beta a & \overline{a^{*-1}} \alpha a^{*-1} \\
			\overline{a} \gamma a & \overline{a} \overline{\beta} a^{*-1}\end{pmatrix}, \quad Y^{\dag}AY = \begin{pmatrix} \beta+\gamma\overline{b} & (\beta + \gamma\overline{b})b+ \alpha + \overline{b}\overline{\beta} \\ \gamma & \gamma b + \overline{\beta}\end{pmatrix}, \quad Z^{\dag}AZ = \begin{pmatrix} -\overline{\beta} & \gamma \\ \alpha & -\beta \end{pmatrix},  $$
			all of which are reverse Hermitian by inspection.
			Also, $\det(X^{\dag}AX)=\beta\overline{\beta} - \alpha \gamma$ due to cancellation.
			One finds that
	$$\delta(X^{\dag}AX) = (a\beta a^{'-1})(a^{*-1}\overline{\beta}\overline{\alpha}) - (a\alpha\overline{a})(a^{*-1}\gamma a^{'-1}) = a\beta a^{'-1}a^{*-1}\overline{\beta}\overline{\alpha} - a \alpha \overline{a}a^{*-1}\gamma a^{'-1}.$$
As $\gamma, \alpha \in \RR$ and for $x \in \CC_n^{\times}$ we have that $x \overline{x} = x^*x^{'} = x^{'}x^* = \vert x \vert^2$. Thus, we obtain
		
		$$\delta(X^{\dag}AX) = \vert a^{-1} \vert^2 \vert a \vert^2 (\vert \beta \vert^2 - \alpha \gamma)= \vert a^{-1} a \vert^2 \delta(A) = \delta(A).$$

		which implies  
$$\delta(Y^{\dag}AY) = (\beta+\gamma\overline{b})(\gamma b + \overline{\beta}) - ((\beta + \gamma\overline{b})b+ \alpha + \overline{b}\overline{\beta})(\gamma)= \beta\gamma b + \vert \beta \vert^2 + \gamma\overline{b}\gamma b - \gamma\beta b - \gamma^2\overline{\beta}\beta - \gamma\alpha - \gamma\overline{b}\overline{\beta} = \delta(A).$$

		we have
		$$\delta(Z^{\dag}AZ) = - \overline{\beta}(-\beta) - \gamma\alpha = \delta(A).$$
	Thus the discriminant is preserved under the action. 
\end{proof}

We can now prove integrality of the Ford spheres.
 \begin{theorem}[Integrality]\label{T:integral}
	Let $g = \begin{psmallmatrix} a & b \\ c & d \end{psmallmatrix} \in \PSL_2(\CC_{n})$. The sphere $g^{-1}S_{\infty}$ has center $p$ and radius $r$ given by  
	$$p = \frac{1}{\vert c \vert^2}\left (-\overline{c}d - \frac{i_n}{2}(c^*b - a^*d)\right), \qquad r =\frac{1}{2\vert c\vert^2}.$$
	In particular if we take $g\in \PSL_2(O)$ we see that the curvature of the sphere is an integer.
\end{theorem}

\begin{proof}
	From Proposition~\ref{L:r-and-p} and Theorem~\ref{T:associated-A}, $A_1 = A_{\infty}^g$ we have that the radius of $S_{A_{\infty}^g}$ is given by 
	$$r = \frac{\sqrt{\vert \beta_1 \vert^2 - \alpha_1\gamma_1}}{\gamma_1}.$$
	
	As the discriminant of a reverse Hermitian matrix is invariant under our conjugation action (Proposition~\ref{lem:conjugation-preserves}), we see 
	$$\delta(A^g) = \vert \beta_1 \vert^2 - \alpha_1\gamma_1 = \delta(A_{\infty} )= 1/4.$$
	Thus, $r = 1/2\vert c \vert^2$
	We similarly obtain the center, $p$, by substituting our computed $\overline{\beta_1}$ and $\gamma_1$ values from Theorem~\ref{T:associated-A}. 
\end{proof}

\begin{remark}
In the arxiv manuscript \cite{Kisil2015} proposes a formula for the radii of $g(S_{\infty})$, which appears to be off by a factor of 2.  At the time of writing, we have been unable to thoroughly compare our calculation with Kisil’s due to unclear (sometimes incorrect) definitions and insufficient detail in \cite{Kisil2015}.
\end{remark}

\subsection{Internal Disjointness}\label{S:internal-disjointness}
We now show that the intersection of the interior of two spheres in $\mathcal{P}$ is empty.
\begin{proposition}
	Every $n$-sphere tangent to $V_n \subset \partial\Hcal^{n+1}$ is determined by its point of tangency and one other point. 
\end{proposition}
\begin{proof}
Let $S$ be a sphere containing points $p$ and $q$ with $p$ the point of tangency.  
One can then inflate the radius of a sphere tangent to $p$ until one reaches the point $q$. 
This determines the radius. 
The radius together with the point of tangency uniquely determine the sphere.
\end{proof}

Let $\Gamma = \PSL_2(O)$ with $O = \quat{-d_1,\ldots,-d_{n-1}}{\ZZ}$. 
Let $K = O\otimes_{\ZZ}\QQ$ and view $\Vec(K) \cup \lbrace \infty \rbrace$ as the boundary of $\Hcal^{n+1,*}$.

\begin{proposition}\label{L:preservation-of-points-preservation-of-spheres}
	Let $S$ be a Ford sphere tangent to the boundary at $a \in \Vec(K)$. 
	Let $\gamma \in \Gamma$.
	If $\gamma(a) = a$ then $\gamma(S)=S$.
\end{proposition}
\begin{proof}
	We know that $S = \alpha(S_{\infty})$ for some $\alpha \in \Gamma$ and hence $\alpha(\infty)=a$.
	We know that $\Stab_{\Gamma}(a) = \alpha\Gamma_{\infty} \alpha^{-1}$ and $\Gamma_{\infty}$ is the generated by  translations and rotations.
	Showing now that $\gamma(S) = S$ for $\gamma\in \Stab_{\Gamma}(a)$  reduces to showing that $S_{\infty} = \partial \Hcal^{n+1} + i_n$ is invariant under translations $\tau_s$ and rotations $\pi_t$ for $s \in\Vec(O)$ and $t\in O^{\times}$ the Clifford group.
	The statement for translations is evident. 
	Similarly, $\pi_t(x) = tx((t^*)^{-1})^{-1}= txt^*$ which is an orthogonal transformation (see \cite[Lemma 10, v]{Waterman1993} or \cite{Dupuy2024}) on $V_n$ and preserves $i_n$.
\end{proof}

\begin{theorem}
	Every Ford sphere $S$ is uniquely determined by its point of tangency $a\in\Vec(K) \cup \lbrace \infty \rbrace$.
\end{theorem}
\begin{proof}
	Suppose that $S$ and $S'$ are both Ford spheres tangent to $\partial \Hcal^{n+1}$ at $a \in K$. 
By transitivity of the action of $\Gamma$ on the set of Ford spheres there exists some $\gamma \in \Gamma$ such that $\gamma(S) = S'$. 
But since $\gamma$ preserves $\partial \Hcal^{n+1}$ we must have $\gamma(a) = a$. 
By the Proposition~\ref{L:preservation-of-points-preservation-of-spheres} we see that $\gamma(S)=S$. This proves $S = S'$.
\end{proof}

\begin{theorem}
	For $x \in \partial \Hcal^{n+1}$ principal, the sphere $S_x$ is invariant under $\Stab_{\Gamma}(x)$. 
\end{theorem}
\begin{proof}
	By pulling back to $x=\infty$ is suffices to prove this for this in the case when $x=\infty$ since all of the other stabilizers are conjugate to this one. 
	In this situation one has $S_{\infty} = i_n + V_n$. 
	By the description of the stabilizer of $\infty$  it suffices to show that $\tau_s$ and $\sigma_t$ stabilize $S_{\infty}$. 
	Clearly $\tau_s$ stabilizes $S_{\infty}$ as it is a translation along $V_n$.
	For $i_n + b \in i_n + V_n$ one has $\sigma_t(i_n+b) = ti_nt^*+tbt^*$. 
	From the Spin representation we know that $b$ is preserved. 
	It suffices to prove this for every Clifford group element $t \in O^*$.
	Viewing $O^* \subset \CC_n^*$ and using that every element of $\CC_n^*$ can be written as a product of Clifford vectors we can write $t =t_1\cdots t_m$ where $t_j \in V_n$.  
	One has $t i_n t^* = t_1 \cdots t_m i_n t_m \cdots t_1 = t_1\cdots t_m \overline{t}_m \cdots \overline{t}_1 i_n = \vert t \vert^2 i_n$. 
	Since $t \in O$ this must be a positive integer and since it is invertible we have $\vert t \vert^2=1$. 
\end{proof}

\begin{theorem}[Internal Disjointness]
Every two distinct spheres $S$ and $T$ in the Ford packing $\mathcal{P}$ have disjoint interiors.
\end{theorem}
\begin{proof}
The idea of the proof can be seen in Figure~\ref{F:h2-spheres}.
There the sphere $S_{\infty}$ is inscribed in $B$ with point of tangency only at $\partial B$. 
Consider two spheres $S$ and $T$. 
Again, without loss of generality we can suppose $S=S_{\infty}$ and $T=S_x=\gamma(S_{\infty})$ where $x=\gamma(\infty)$.
There are three cases. 
The first case is that $\gamma \in \Gamma_{\infty}$ in which case $S=T$ which is not considered in the proposition.
If $x\neq \infty$ we can translate $x \in \partial \Hcal^{n+1} = V_n$ back to the fundamental domain $F\subset V_n$ for the additive group $\Vec(O)$ and suppose that $x \in F$.
We now have $S_x\cap S_{\infty} \subset \gamma(\overline{B}) \cap \overline{D}$.
Since $\overline{B}$ tile the plane and $\overline{B} \neq \gamma(\overline{B})$ we must have that $\gamma(\overline{B}) \cap \overline{D}  \subset \partial D$. 
This establishes that the spheres have disjoint interiors.
\end{proof}

\subsection{Connectivity and Adjacency}\label{S:connectivity}

Let $\mathcal{P}$ be the collection of Ford spheres in $\Hcal^{n+1}$ for $\Gamma = \PSL_2(O)$ for some order $O$ in $K = \quat{-d_1,\ldots,-d_{n-1}}{\QQ} \subset \CC_n$. 
\begin{definition}
We define the \emph{tangency graph} of $\mathcal{P}$ by declaring the vertices to be $\mathcal{P}$ and declaring that spheres $S$ and $T$ are adjacent if and only if $S \cap T \neq \emptyset$.
\end{definition}

\begin{definition}
	We say that packing $\mathcal{P}$ is \emph{connected} if and only if the tangency graph is connected.
\end{definition} 

We will make use of the definition of Clifford-Euclideanity. 
\begin{definition}
Let $O$ be an order in a rational Clifford algebra $K$ associated to a positive definite integral quadratic form.
We say that $O$ is \emph{Clifford-Euclidean} if and only if for all $c,d\in O^{\mon}$ with $c^{-1} d \in K$ invertible there exists some $a\in \Vec(O)$ such that $d=ca+r$ and  $\vert r \vert<\vert c\vert.$
\end{definition}
The basic properties and algorithms regarding Clifford-Euclidean orders are summarized in \cite[\S 3.5]{Dupuy2024}.

\begin{theorem}[Connectivity]\label{T:connectivity}
Let $O$ be an order in $\CC_n$ which is Clifford-Euclidean with respect to the norm. 
The Ford spheres $\mathcal{P}$ for $\PSL_2(O)$ are connected. 
\end{theorem}

\begin{proof}
We claim that for every $S \in \mathcal{P}$ there is a finite walk on the tangency graph to $S_{\infty}$.  
The proof is by induction on the curvature of $S$ (which is an integer) where we define the curvature of $S_{\infty}$ to be zero. 
This is the base case where there is a walk of length zero.

By induction it is enough to show that for any Ford sphere $S_{\alpha}$, there exists some element $\gamma \in\PSL_{2}(\mathcal{O})$ that takes $S_{\alpha}$ to an adjacent Ford sphere $\gamma(S_{\alpha})$ and $\gamma(S_{\alpha})$ has a larger radius.

Consider a Ford sphere $S_{\alpha}$. 
By the definition of a Ford sphere, $S_{\alpha}=g(S_{\infty})$ for some $\gamma\in \Gamma=\PSL_{2}(O)$. 
Now consider some other Ford sphere $S_{\omega}$ which is of radius $1/2$ and is tangent to $S_{\infty}$. 
Now if we take the same $g$ as above and consider $S_{\beta}=g(S_{\omega})$ we find that $S_{\beta}$ is tangent to $S_{\alpha}$, since $S_{\omega}$ was tangent to $S_{\infty}$.

We claim that we can choose $\omega \in \Vec(O)$ so that the radius of $S_{\beta}$ is greater than that of $S_{\alpha}$. 
If $S_{\alpha} = g(S_{\infty})$ with $g=\begin{pmatrix} a & b \\ c & d\end{pmatrix} \in \Gamma$ then $\alpha = ac^{-1}$. 
Now applying this same $g$ to $\omega$ we have that $S_{\beta}=g(S_{\omega})= S_{g(\omega)}$ and hence it has radius $1/(2\vert c\omega+d\vert^{2})$. 
By the Clifford-Euclidean property of $O$ there exists some $\omega\in \Vec(O)$ such that $\vert c^{-1} d + \omega\vert<1$ so that $\vert c \omega + d \vert < \vert c \vert$ which implies $r(\beta)>r(\alpha)$. 
\end{proof}

At the time of writing this it is unclear if there are infinitely many orders $O$ which are Clifford-Euclidean. 
See \cite[\S 17.1]{Dupuy2024} for this and a collection of open problems related to this.

\subsection{Diophantine Approximation}
Just as in Ford's original article \cite{Ford1918} there is a very simple proof of Dirichlet's Theorem.
\begin{theorem}
	Let $O$ be an order in $\CC_n$ which is Clifford-Euclidean with respect to the reduced norm.
	Let $K$ be the rational Clifford algebra containing $O$.  
	For every $\alpha \in V_n\setminus \Vec(K)$ there are infinitely many $(c,d)$ unimodular such that 
	 $$ \left \vert \alpha - c^{-1}d  \right \vert < \frac{1}{2\vert c \vert^2}. $$ 
\end{theorem}
\begin{proof}
	By Clifford-Euclideanity of $O$ the rational Clifford vectors $\Vec(K)$ are dense in $\Vec(O)$. 
	For $\alpha \in V_n \setminus \Vec(K)$ the line $\lbrace \alpha+t i_n \colon t>0 \rbrace$ intersects infinitely many $S_{c^{-1}d} \in \mathcal{P}$. 
	This means that $\alpha$ is within the radius of $S_{c^{-1}d}$ from $c^{-1}d \in \Vec(K)$. 
	The formula for the radius proves the result.
\end{proof}

\subsubsection{Farey Fractions}
The analogs of Farey fractions is evident. 
Suppose that $O$ is Clifford-Euclidean with rational Clifford algebra $K$.
\begin{definition}
We will say that $x,y \in \Vec(K)$ are \emph{Farey adjacent} if and only if there exists some $g=\begin{psmallmatrix} a & b \\
c & d \end{psmallmatrix} \in \SL_2(O)$ such that $g(\infty)=x$ and $g(0)=y$.
\end{definition} 

Using out definition we can prove the following properties. 
\begin{proposition}
	\begin{enumerate}
		\item If $x$ and $y$ are Farey adjecent then $S_x$ and $S_y$ are tangent.
		\item If $z$ is the mediant of $x$ and $y$ then $S_z$ is adjacent to both $S_x$ and $S_y$.
	\end{enumerate}
\end{proposition}
\begin{proof}
Write $g=\begin{psmallmatrix} a & b \\
	c & d \end{psmallmatrix} \in \SL_2(O)$ with $g(\infty)=ac^{-1}=x$ and $g(0)=bd^{-1}=y$.
We know that $S_0$ and $S_{\infty}$ are tangent. 
Their images under $g$ are $S_x$ and $S_y$ and the image of adjacent spheres are adjacent. 

The mediant of $ac^{-1}$ and $bd^{-1}$ is $z=(a+b)(c+d)^{-1}$ which is $g(1)$.
We also know that $S_1$ is adjacent to both $S_0$ and $S_{\infty}$. 
Since the images of adjacent sphers are adjacent we know that $S_{z}$ will be adjacent to both $S_x$ and $S_y$. 
\end{proof}

\subsubsection{Continued Fraction}
We report that a theory of continued fractions for Clifford-Euclidean domains exists in the Clifford-Bianchi setting. 
A detailed treatment will appear in forthcoming work of Dupuy and Logan on modular symbols in the Clifford-Bianchi setting.
A \textsf{Magma} package containing this functionality is available at 
\begin{center}
    \url{https://github.com/tdupu/magma-clifford-algebras}.
\end{center}

\section{Discussion}\label{S:discussion}

\subsection{History of Clifford-algebraic $\SL_2(\CC_n)$}

An excellent history of the subject Clifford algebraic M\"{o}bius transformations up to the year 1985 is given by Ahlfors in \cite{Ahlfors1985}, which we repeat.
The group $\SL_2(\CC_n)$ was introduced by Vahlen in \cite{Vahlen1902} and then forgotten until Maass \cite{Maass1949} and then promptly forgotten again.
In \cite{Ahlfors1984}, Ahlfors does mention an unfavorable encyclopedia entry by E. Cartan and Study.
lso, \cite{Ahlfors1985} mentions that the foundational papers of Fueter in 1926 seem to be unaware of the papers of Vahlen which explains why the subject of quaternionic analysis (see \cite{Gentili2022}) seems to have developed independently. 

In the 1980s and early 1990s, building on work of Ahlfors, there are the works of Maclachlan, Weiland, and Waterman \cite{Maclachlan1989} and Waterman \cite{Waterman1993} establishing some basic properies of Clifford-Bianchi groups and the general theory of M\"{o}bius transformations. 
Perhaps most substantially during this period are the papers of Elstrodt,Grunnewald, and Mennicke \cite{Elstrodt1987, Elstrodt1988,Elstrodt1990} which develop the basic theory and relationships to spin groups, culminating in a proof of Selberg's conjecture on the smallest size of the hyperbolic Laplacian (another independent proof of this result that doesn't use Clifford algebras was given in \cite{Li1987}). 
Many important results were established during this period. 
The notation for $\SL_2(\CC_n)$ that we use is a combination of notation used by Ahlfors and Elstrodt-Grunnewald-Mennicke.

Following this period there is work of Vulakh \cite{Vulakh1993,Vulakh1995,Vulakh1999}, primarily concerned with generalizing Diophantine approximation results related to Markov and Lagrange spectra to the Clifford setting (Vulakh is an expert in approximation of complex numbers by $p/q$ with $p,q$ elements of an order in $\QQ(i\sqrt{m})$ for $m>0$ squarefree).

Finally, there is the work of Krau\ss har on extending some classical complex analytic theory of modular forms to a Clifford analytic setting culminating in a book \cite{Krausshar2004}.
Since then there have been the papers of McInroy \cite{McInroy2016}, Zemel \cite{Zemel2021b,Zemel2021} which seek to analyze Clifford-Bianchi groups for general quadratic modules.\footnote{
	We thank McInroy for pointing out the papers of Zemel to us. 
	These papers appear to have overlap with the \S2 of \cite{Dupuy2024} concerning group schemes and arithmetic Bott periodicity.}

The monograph \cite{Dupuy2024}, gives a higher dimensional version of Swan's paper \cite{Swan1971} for Bianchi groups. It gives a self-contained treatment of the basic material above. It also provides a number of new orders, explicit descriptions of fundamental domains for Clifford-Bianchi groups, as well as algorithms for orders and fundamental domains.

It is curious to note that Clifford M\"{o}bius transformations appear in 1902 \cite{Vahlen1902} and the problem of continued fractions appears as early as 1852 \cite{Hamilton1852,Hamilton1853} but that serious combining these two areas seem to only appear in 1999 \cite{Vulakh1999}.
These tools were available to Ford, whose original paper was in \cite{Ford1916}. 
We can't answer for Ford, but Elstrodt, Grunnewald, and Mennicke offer an explanation as to why after 1935 nobody looked for special isomorphisms with $\SO_{1,n}(\RR)^{\circ}$ (this comes from \S6 of \cite{Elstrodt1987}): van der Waerden (see  \cite{Hofreiter1935}) developed methods for giving isomorphism between classical Lie groups; this method was essentially exhausted by Dieudonn\'{e} in \cite{Dieudonne1952} and \cite{Dieudonne1963} and using his methods he only gave special isomorphisms involving $\O_{1,n}(\RR)$ for $3\leq n \leq 6$.
These results seem to indicate that the isomorphisms $\PSL_2(\RR) \cong \SO_{1,2}(\RR)^{\circ} \cong \Isom(\Hcal^2)^{\circ}$ and $\PSL_2(\CC) \cong \SO_{1,3}(\RR)^{\circ} \cong \Isom(\Hcal^3)^{\circ}$ are exceptional and can't be extended. 
Dieudonn\'{e}'s influence and authority likely only furthered this belief.
The isomorphisms are still believed to be accidental by many people (see e.g. \cite[Remark 5.1]{Sheydvasser2023}).

\subsection{Apollonian, Baragar, Boyd-Maxwell, and Other Packings}
 This section gives context for our construction in the existing sphere packing and arithmetic hyperbolic manifold literature. This history is quite vast (dating back to Apollonius of Perga) and this subject is rife with misattributions and rediscoveries (presumably due to the large volume of papers that it covers). The idea of generalizing the Ford circles is not new, although the connection to orders in Clifford algebras seems to be new.

\emph{Classical Apollonian sphere packings} $\mathcal{P}$ are a collection of spheres $\mathcal{P}$ in $\RR^{n}$ containing a particular $(n+2)$-cluster $\lbrace S_{0,1},\ldots,S_{0,n+1}\rbrace$ and having the property that for every $(n+1)$-cluster $\lbrace S_1,\ldots, S_{n+1} \rbrace$ contained in $\mathcal{P}$ and every sphere $S'$ such that $S'$ is tangent to $S_i$ for $i=1,\ldots,n+1$ that $S' \in \mathcal{P}$. 
We say that the base $(n+2)$-cluster \emph{generates} $\lbrace S_{0,1},\ldots,S_{0,n+1}\rbrace$ the Apollonian sphere packing. 

The classical Apollonian circle packing (also called the Apollonian gasket) containing the Ford circles is generated by the lines $y=0$,$y=1$ and the circles $C_0$ and $C_1$ of radius $1/2$ tangent to the points $(x,y)=(0,0)$ and $(x,y)=(1,1)$ on the $x$-axis.
Given the close connection between the Apollonian gasket and Ford circles, it is natural then to seek higher dimensional versions of the Apollonian gasket in the literature which could contain the Ford spheres defined in this paper.

Based on some work of Boyd \cite{Boyd1974} on higher dimensional Apollonian packings it was believed that higher dimensional spheres Apollonian packings don't exist.  
Such an assertion of non-existence of these packings can't be found directly in the paper of Boyd \cite{Boyd1974}, but appears in the final sentence of the AMS Mathematical Review for \cite{Boyd1974}, MR350626, by Wilker (see also \cite[p. 356]{Lagarias2002}).

This claim is actually false. 
There exist Apollonian packings in higher dimensions. 
In \cite{Baragar2011}, Baragar was been able to associate sphere packings to K3 surfaces $X$--- the connection between K3 surfaces and hyperbolic geometry comes through the Hodge index theorem to show that intersection form on the Picard group of $X$ has signature $(1,\rho-1)$ where $\rho$ is the Picard rank of the $X$. 
In \cite{Baragar2017} Baragar shows that the classical Apollonian circle packing corresponds to a K3 surface. 
He was able to give examples in dimensions 4,5,6 in \cite{Baragar2018}.
The packings in dimensions 7 and 8 are dealt with in \cite{Baragar2022}.
(We note that the general projective geometry setup for studying packings in this way is described very well in \cite[\S 4]{Dolgachev2016}.)

For us, these packings of Baragar seemed particularly interesting given the close connections between Clifford algebras and the Kuga-Satake construction (see \cite[\S5.7]{Dupuy2024}) and the fact that classical Ford circles are contained in the Apollonian gasket.

Interestingly, the packings associated to K3 geometry in dimensions 4,5,6 of \cite{Baragar2018} were actually a rediscovery of packings of Maxwell \cite{Maxwell1982}. We learned this from Baragar, see \cite[pg 2]{Baragar2023}.

In the literature these packings are sometimes Boyd-Maxwell packings (see \cite{Chen2015}). 
A \emph{Boyd-Maxwell} packing is a collection of spheres in hyperbolic $n$-space associated to Coxeter groups\footnote{Coxeter groups are finitely presented group generated by reflections $r_i$ with relations generated by relations of the form $(r_ir_j)^{m_{ij}}=1$.} which admit a Lorentzian realization (see also \cite[\S5]{Dolgachev2016}).
Note that these groups associated to these packings are hyperbolic reflection groups.
It is a natural question then to ask how our packings are related to Boyd-Maxwell packings.

Sheydvasser shows in \cite[\S15]{Sheydvasser2019} shows that for many orders in quaternion algebras $\PSL_2(O)$ can be embedded in hyperbolic reflection groups.
This is not a general phenomena. 
There cannot be infinitely many dimensions such that there exists an order $O \subset \CC_n$  such that $\PSL_2(O)$ is contained in an arithmetic reflection groups for $\Hcal^{n+1}$ --- this is simply because it is a theorem of Vinberg \cite{Vinberg1981} there are only finitely many arithmetic hyperbolic reflection groups (see the survey \cite{Belolipetsky2016}).

\begin{question}\label{Q:baragar}
	Which Boyd-Maxwell packings contain a Ford packing?
\end{question}

It is tempting to ask: ``If Ford packing for $O$ is contained in a Boyd-Maxwell packing then is $O$ necessarily Clifford-Euclidean?''.
Note that Sheysvasser shows in \cite[Lemma 17.4 (and the discussion that preceeds)]{Sheydvasser2019} shows that certain three dimensional packings associated to $\PSL_2(O)$ for $O$ an order in a quaternion algebra (which act on $\Hcal^4$!) other than the three Clifford-Euclidean exceptions are not Boyd-Maxwell. 
Also, there are restrictions on dimensions associated to these sorts of packings as they are only packings for Coexter diagrams of level $\leq 2$ which only exist in dimension $\leq 10$ (\cite[p. 33]{Dolgachev2016}).

In \cite{Baragar2023} Baragar gives packings which are not Boyd-Maxwell but Apollonian in a certain sense.\footnote{
	Baragar says a collection of spheres in $\Hcal^{n}$ has the \emph{weak Apollonian property} if it has a cluster of $n+2$ spheres and has the \emph{Apollonian property} if every sphere is contained in a cluster of $n+2$ spheres. 
	We call a packing \emph{Apollonian} if it is crystallographic (i.e. generated by reflections), maximal, and has the Apollonian property.
}
\begin{question}\label{WQ:E8-packing}
	It the Enriques packing of \cite{Baragar2023} the natural extension of the Ford packing for $O_{E_8}$ where $O_{E_8}$ is the order in $\CC_8$ described in \cite{Dupuy2024} with $\Vec(O_{E_8})$ an $E_8$-lattice?
\end{question}

Exploring the literature futher, we find there is a large body of work motivated by the problem of finding dense packings (this question is posed in \cite{Coxeter1954}) and the problem of finding minimal volume arithmetic hyperbolic manifolds (existence is proved in  \cite{Kazdan1968}). 
A search on the articles citing \cite{Coxeter1954} and \cite{Kazdan1968} related to hyperbolic sphere packing yields a staggering number of results.
We make no attempts at summarizing recent developments and mention that the literature in this area is vast.\footnote{For example Szirmai boasts 20 publications related to the density of hyperbolic sphere packings between 2014 and 2018.}
Both of these threads give a large number of a hyperbolic manifolds which could be related to our groups $\PSL_2(O)$.

Our groups $\PSL_2(O)$ are arithmetic and there are infinitely many of them, but perhaps some special subclass of orders have a relationship to arithmetic reflection groups. 
For example, a result of Hild \cite{Hild2007} shows that in dimension less than 10, the cusped, complete, hyperbolic orbifolds of minimal volume all have a minimal lattice packing the boundary. 
This is interesting in relation to \cite{Dupuy2024}: for certain Clifford-Bianchi orders the lattices $\Vec(O)$ appear given minimal lattice packings at the boundary.
The statement that these groups are cuspidally principal is also consistent with the observations from \cite{Dupuy2024} (see \cite[\S3.5]{Dupuy2024}).

\begin{question}
	For each group $\Gamma$ of \cite{Hild2007} does there exist some $O$ such that $\PSL_2(O)<\Gamma$? 
\end{question}

\subsection{Soddy-Gossett Formula}
Descartes formula or the Soddy-Gossett formula is a formula that relates curvatures of spheres.\footnote{The formula is neither due to Soddy nor Gossett and there is a long history of the formula which is summarized in \cite{Lagarias2002}.}
This formula is central to the representation problem of sphere packings which investigates problems related to the distribution of curvatures that appear in integral collections of tangent spheres (see for example \cite{Graham2003},\cite{Eriksson2007},\cite{Kontorovich2019b}).

The Soddy-Gossett formula in higher dimensions, states that if there are $n+2$ mutually tangent spheres in $n$-space then the curvatures (or bends) $b_i=1/r_i$ of the spheres are related by 
\begin{equation}\label{E:descartes}
 (\sum_{i=1}^{n+2} b_i)^2  = n\sum_{i=1}^{n+2}b_i^2.
\end{equation}
See \cite[Theorem 1.2]{Lagarias2002}.
We have the natural question then
\begin{question}\label{Q:descartes}
	For which packings in $\Hcal^n$ do there exist clusters of $n+2$ mutually adjacent spheres? Does a Soddy-Gossett formula hold for these Ford packings? (This could be \eqref{E:descartes} or a modification of  \eqref{E:descartes} in the style of Guettler-Mallows \cite{Guettler2010}.)
\end{question}

\subsection{Superpackings and Diophantine Approximation}
Ford spheres have a deep connection to continued fractions and Diophantine approximation.
Ford's original papers\footnote{Which were based on his PhD thesis \cite{OConnor2007}.} \cite{Ford1916} and \cite{Ford1918} were motivated by Hurwitz's theorem on optimal constants for approximation of an irrational number by a rational number. 
In \cite{Ford1925} he states that proofs of E.Borel \cite{Borel1903} and Humbert \cite{Humbert1916} are inadequate for generalization to the complex numbers and developed his technique citing work of Poincar\'{e} \cite{Poincare1900} and Bianchi \cite{Bianchi1891}. 

\begin{figure}[htbp!]
	\begin{center}
		\begin{tabular}{cc}
			\includegraphics[scale=0.5,trim = 4cm 3cm 0 0]{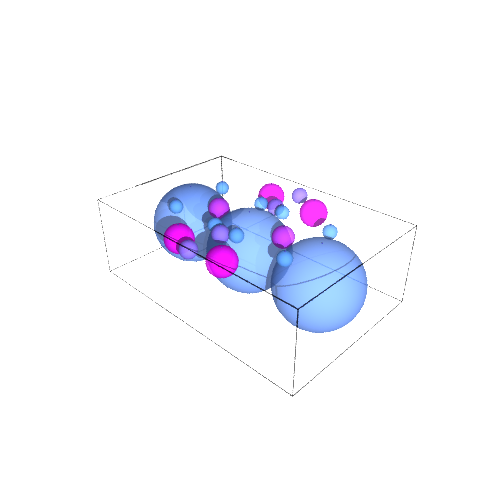} & \includegraphics[scale=0.5,trim = 0 0 0 0]{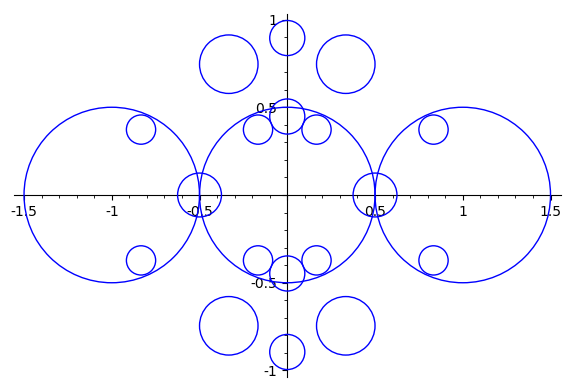} \\
		\end{tabular}
	\end{center}
	\caption{The Ford packing for $\PSL_2(\ZZ[i\sqrt{5}])$.
		Note that this maximal order has class number 2 and hence there are two connected components of this packing. Left: A picture of the Ford packing in $\Hcal^3$. Right: the projection of the equators of the Ford packing. \label{F:ford-sqrtm5} }
\end{figure} 

Schmidt \cite{Schmidt1967}, \cite{Schmidt1975} introduced the notion of Farey triangle which allowed for the analysis of optimal constants in Hurwitz' inequality. 
There has been some extensions of this work in the Clifford-Bianchi setting by Vulakh \cite{Vulakh1995},\cite{Vulakh1999} (many of the results here depended on general considerations of Riemannian geometry).
Given that we know more about the arithmetic of the orders in Clifford algebras, it may be useful to revisit \cite{Vulakh1995,Vulakh1999}.
We point to \cite{Martin2023} which provides estimated on approximation of complex numbers by elements of imaginary quadratic fields whose orders are not Euclidean which could possibly be generalized to higher dimensions.

Graham, Lagarias, Mallow, Wilks, Yan papers \cite{Graham2005},\cite{Graham2006},\cite{Graham2006b} define, classify, and study superpackings (which was then picked up on by, say, Kontorovich-Nakamura \cite{Kontorovich2019} and many other authors). Following \cite{Stange2018b} we can make the following definition for $O$ an order in a rational Clifford algebra associated to a postiive, definite, primitive, integral quadratic $(n-1)$-form $q$.

In what follows we will abusively conflate $V_m$ with $V_m \cup\lbrace \infty \rbrace$.
\begin{definition}
	A \emph{Schmidt arrangement} for $O$ in 
	$$ \lbrace (ax+b)(cx+d)^{-1} \colon x\in V_{n-1} \cup \lbrace \infty \rbrace, \quad  \begin{pmatrix} a & b \\ c & d \end{pmatrix} \in \SL_2(O) \rbrace. $$
\end{definition}
In the case of $\Hcal^3$ it was proved that these are a superpacking in the sense of \cite{Graham2006} and that, for example, the arrangement is connected if and only if $\Ocal_{K}$ is a Euclidean domain \cite[Theorem 1.5]{Stange2018b}. 
For quaternion orders these were investigated by Sheydvasser in \cite{Sheydvasser2019}. 

It is natural to ask then:
\begin{question}
	In Clifford-Euclideanity characterized by connectivity of the Schmidt arrangement or connectivity of the Ford packing?
\end{question}
For example see Figure~\ref{F:ford-sqrtm5} which shows that $\Ocal_{\QQ(\sqrt{-5})}$ is not connected. 
We note that \cite{Sheydvasser2019} has investigated this problem in the case of Clifford-Euclidean and certain non Clifford-Euclidean orders in quaternion algebras.

\bibliographystyle{alpha}
\bibliography{quat}

\end{document}